\documentclass[12pt]{amsart}
\usepackage{amstext,amsfonts,amssymb,amscd,amsbsy,amsmath,verbatim, mathrsfs, fullpage}
\usepackage[alphabetic,abbrev,lite]{amsrefs} 
\usepackage{ifthen,tikz}
\usepackage{color}
\usepackage{amsthm}
\usepackage{latexsym}
\usepackage[all]{xy}
\usepackage{enumerate}
\usepackage{mathtools}
\numberwithin{equation}{section}

\newtheorem{lemma}[equation]{Lemma}

\newtheorem{prop}[equation]{Proposition}
\newtheorem{cor}[equation]{Corollary}

\newtheorem{thm}[equation]{Theorem}
\newtheorem{notation}[equation]{Notation}

\newtheorem{question}[equation]{Question}

\theoremstyle{definition}
\newtheorem{defn}[equation]{Definition}
\newtheorem{setup}[equation]{Setup}
\newtheorem{example}[equation]{Example}

\theoremstyle{remark}
\newtheorem{remark}[equation]{Remark}
\newtheorem{remarks}[equation]{Remarks}

\newtheorem*{claim*}{Claim}
\newtheorem*{case*}{Case}

\newcommand{\cC}{\mathcal{C}}

\newcommand{\m}{\mathfrak m}
\newcommand{\PP}{\mathbb P}

\renewcommand{\P}{\PP}

\newcommand{\ZZ}{\mathbb Z}

\newcommand{\im}{\operatorname{im}}

\newcommand{\id}{\operatorname{id}}

\newcommand{\Pic}{\operatorname{Pic}}

\newcommand{\Tor}{\operatorname{Tor}}

\newcommand{\Hom}{\operatorname{Hom}} 

\newcommand{\cO}{{\mathcal O}}
\newcommand{\kk}{{\bf k}}

\newcommand{\Sym}{\operatorname{Sym}} 

\newcommand{\CC}{\mathbb{C}}

\newcommand{\defi}[1]{\textsf{#1}} 

\newcommand{\beq}{\begin{displaymath}}
\newcommand{\eeq}{\end{displaymath}}

\def\nc{\newcommand}
\def\on{\operatorname}
\nc{\Q}{\mathbb{Q}}
\nc{\RR}{\mathbf{R}}
\nc{\LL}{\mathbf{L}}
\nc{\xra}{\xrightarrow}
\nc{\xla}{\xleftarrow}
\def\a{\alpha}
\def\om{\omega}

\def\DM{\operatorname{DM}}

\def\th{\on{th}}

\nc{\into}{\hookrightarrow}
\nc{\onto}{\twoheadrightarrow}
\nc{\OO}{\mathcal{O}}
\nc{\Z}{\mathbb{Z}}
\nc{\cA}{\mathcal{A}}
\nc{\w}{\widehat}
\nc{\End}{\on{End}}
\nc{\res}{\frac{1}{x_0x_1}}
\nc{\tF}{\widetilde{F}}
\nc{\tG}{\widetilde{G}}
\nc{\tf}{\widetilde{f}}
\nc{\Com}{\on{Com}}

\nc{\G}{\mathbb{G}}
\nc{\cG}{\mathcal{G}}
\nc{\cE}{\mathcal E}
\nc{\cF}{\mathcal F}
\nc{\cR}{\mathcal R}
\nc{\cD}{\mathcal D}
\nc{\cB}{\mathcal B}
\nc{\cT}{\mathcal T}
\nc{\cL}{\mathcal L}

\nc{\bM}{\mathbf M}
\nc{\bN}{\mathbf N}
\nc{\U}{\mathbf U}
\nc{\BM}{\mathbf B \mathbf M}
\nc{\Dsg}{\on{D}_{\on{sg}}}
\nc{\fC}{\mathcal{C}}
\nc{\fG}{\mathcal{G}}
\nc{\N}{\mathbb{N}}

\nc{\del}{\partial}
\nc{\cone}{\on{cone}}
\nc{\D}{\on{D}_{\on{diff}}}
\nc{\DMb}{\on{D}^b_{\DM}}
\nc{\Db}{\on{D}^{\on{b}}}
\nc{\Kb}{\on{K}^{\on{b}}}
\nc{\fm}{\mathfrak{m}}
\nc{\Flag}{\on{Flag}}
\nc{\DMmin}{\DM_{\on{min}}}
\nc{\Ddiff}{\on{D}_{\on{diff}}}
\nc{\Dbdiff}{\on{D}^\on{b}_{\on{diff}}}
\nc{\wO}{\widehat{\OO}}
\nc{\wT}{\widehat{T}}
\nc{\from}{\leftarrow}
\nc{\wLL}{\widetilde{\LL}}
\nc{\augCech}{\widetilde{\cC}}
\nc{\Fold}{\on{Fold}}
\nc{\Ext}{\on{Ext}}
\nc{\FF}{\mathbf{F}}
\nc{\Comper}{\Com_{\on{per}}}
\nc{\Unfold}{\on{Unfold}}
\nc{\intHom}{\underline{\Hom}}
\nc{\Ex}{\on{Ex}}
\nc{\tg}{\widetilde{g}}

\def\b{\beta}
\nc{\B}{\mathcal{B}}
\nc{\K}{\mathcal{K}}
\nc{\kos}{\on{Kos}}
\nc{\Perf}{\on{Perf}}
\nc{\tR}{\widetilde{\cR}}
\nc{\X}{\mathcal{X}}
\nc{\Cl}{\on{Cl}}
\nc{\fU}{\mathcal{U}}
\nc{\bU}{\mathbf U}

\nc{\st}{\on{st}}

\nc{\coh}{\on{coh}}

\def\D{\mathcal{D}}

\nc{\tU}{\U}
\nc{\bC}{\mathbf{C}}
\nc{\aux}{\on{aux}}
\def\d{\mathbf{d}}
\def\phi{\varphi}

\title{Linear syzygies of curves in weighted projective space}
\author{Michael K. Brown}
\address{Department of Mathematics, Auburn University, Auburn, AL}
\email{mkb0096@auburn.edu}

\author{Daniel Erman}
\address{Department of Mathematics, University of Wisconsin, Madison, WI}
\email{derman@math.wisc.edu}
\date{\today}
\thanks{The first author was supported by NSF-RTG grant 1502553.  The second author was supported by NSF grant 
DMS-2200469}

\begin{document}

\maketitle

\begin{abstract}
We develop analogues of Green's $N_p$-conditions for subvarieties of weighted projective space, and we prove that such $N_p$-conditions are satisfied for high degree embeddings of curves in weighted projective space. A key technical result links positivity with low degree (virtual) syzygies in wide generality, including cases where normal generation fails.
\end{abstract}

\section{Introduction}

One of the foundational results connecting syzygies and geometry is Mark Green's theorem on linear syzygies of smooth curves:

\begin{thm}[\cite{green2}]
\label{thm:green}
Let $C$ be a smooth curve of genus $g$ embedded in $\PP^n$ via a complete linear series $|L|$ and $F$ the minimal free resolution of the homogeneous coordinate ring of $C$. If $\deg(L) \ge 2g + 1 + p$ for some $p \ge 0$, then the embedding $C \into \PP^n$ satisfies the $N_p$-condition: that is, it is normally generated, and $F_i$ is generated in degree $i+1$ for $1 \le i \le p$. 
\end{thm}

Let us recall the definitions of the terms in the theorem.  The embedding $C \into \PP^n$ is defined to be \defi{normally generated} (or \defi{projectively normal}) if the section ring $\bigoplus_{i \ge 0} H^0(C, L^i)$ is generated in degree 1. Theorem~\ref{thm:green} gives a common generalization, in the language of syzygies, of two classical results showing that the algebraic presentation of a curve exhibits more rigid structure as its degree grows. Specifically: the $p = 0$ case of Theorem~\ref{thm:green} is Castelnuovo's Theorem, which says that $C\subseteq \PP^r$ is normally generated if $\deg(L)\geq 2g+1$~\cite{castelnuovo}; while the $p=1$ case is a result of Fujita and Saint-Donat stating that $C$ is cut out by quadrics whenever $\deg(L) \geq 2g+2$ \cite{fujita,SD}.

Theorem~\ref{thm:green} helped launch the modern study of the geometry of syzygies.
It led to numerous generalizations involving embeddings of surfaces~\cite{AKL, banagere-hanumanthu,gp3, gallego-purna2,KL, chung-yeung}, smooth higher dimensional varieties~\cite{eghp,ein-lazarsfeld-np,gallego-purna,hss, hwang-to}, abelian varieties~\cite{chintapalli,pareschi,ppI,ppII}, and Calabi-Yau varieties~\cite{niu}; see also~\cite{dao-eisenbud, ein-niu-park}. 
Theorem~\ref{thm:green}  also led to Green's Conjecture, which proposes a relationship between the Clifford index of a non-hyperelliptic curve over $\CC$ and the linearity of the free resolution of its coordinate ring with respect to its canonical embedding \cite{greenII}. This conjecture remains open in general and is a highly active area of research, see e.g. \cite{AFPRW, AV,  voisin-even, voisin-odd}. For an introduction to the wide circle of ideas on syzygies of curves, see Aprodu-Nagel's book~\cite{aprodu-nagel}, Ein-Lazarsfeld's survey~\cite{EL-bulletin}, Eisenbud's book \cite{eisenbud}, and more~\cite{EL-large,farkas}.

The past two decades have seen a flurry of activity devoted to generalizing 
work on syzygies to the \emph{non}standard graded setting. For instance, Benson~\cite{benson} generalized Eisenbud-Goto's Theorem on Castelnuovo-Mumford regularity and linear free resolutions~\cite[Theorem 1.2(1)]{EG} to nonstandard $\Z$-graded polynomial rings, leading to breakthroughs in invariant theory due to Symonds \cite{symonds2, symonds}.
Benson's resulting ``weighted" analogue of Castelnuovo-Mumford regularity was generalized by Maclagan-Smith to multigraded polynomial rings in \cite{MS}, with a view toward applications in toric geometry. This in turn led to much followup work on multigraded regularity \cite{BC, BHS, BHS2, chardin-holanda, chardin-nemati}, as well as a wide-ranging program on multigraded syzygies: \cite{tate, BES, bkly, brownshort, buse-chardin, EES, HNV, HS, hss, SV, yang}.

The present work is the first generalization of Green's Theorem that allows the \emph{target} of the embedding to be a variety other than projective space, connecting with the literature on nonstandard gradings  discussed above, and raising many new questions.
For instance, it is far from obvious how to even state an analogue of Theorem~\ref{thm:green} for curves embedded in weighted projective space.
To define weighted $N_p$-conditions for $p \ge 1$, one must ask: what does it mean for a complex of free modules over a nonstandard-$\Z$-graded polynomial ring to be \emph{linear}? To illustrate the subtlety in this question, take a standard graded polynomial ring $S$ and an $S$-module $M$ generated in degree 0.
The minimal free resolution $F$ of $M$ is linear if and only if it satisfies the following three equivalent conditions:
\begin{enumerate}
\item The differentials in $F$ are matrices of linear forms. 
\item The Betti table of $F$ has exactly one row.
\item The degrees of the syzygies 
grow no faster than those of the Koszul complex.
\end{enumerate}

In the nonstandard $\Z$-graded case, each of these yields a {\em distinct} analogue of a linear resolution.  And there is no obvious ``best'' choice, as each measures something meaningful:
(1) leads to strong linearity (Definition~\ref{def:strongly linear}) and the BGG correspondence as in \cite{linear}, (2) to weighted regularity and local cohomology as in \cite{benson}, and (3) to {Koszul linearity} (Definition~\ref{defn:koszul}) and connections with Koszul cohomology; see \S\ref{sec:koszul linearity} for details on these notions and how they are related.
In fact, a central obstacle in our work is the technical challenge of interpolating between these non-equivalent notions of linearity, a challenge which simply is not present in the classical setting.

Phrasing an analogue of Theorem~\ref{thm:green} also requires weighted versions of complete linear series and of normal generation. 
A weighted notion of normal generation is fairly straightforward---see Definition~\ref{defn:proj normal}. But defining a weighted version of complete linear series turns out to be rather subtle,
and as with linearity, there are multiple potential analogues, depending on which aspect of the classical notion one considers. 
The subtlety arises partly because there is a tension between nondegeneracy and normal generation (see Examples~
\ref{ex:naive} and ~\ref{ex:not normally generated}), and partly because any analogue must depend on data beyond just the line bundle.
We propose \defi{log complete series} in Definition~\ref{defn:log complete} as an analogue of complete linear series that requires a minimal amount of extra data: a base locus and a degree.  These lead to a rich family of examples of embeddings, where the underlying weighted spaces are fairly simple, involving just two distinct degrees.

For our weighted $N_p$-conditions, we use condition (3) from above:

\begin{defn}\label{defn:1regular}
Let $S$ be the $\ZZ$-graded polynomial ring corresponding to a weighted projective space $\PP(W)$.  
Write $w^i$ for the maximal degree of an $i^{\th}$ syzygy of the residue field. 
Let $Z \subseteq \PP(W)$ be a variety and $F=\left[F_0\gets F_1 \gets \cdots\right] $ the minimal $S$-free resolution of its coordinate ring. We say $Z\subseteq \mathbb  \PP(W)$ satisfies the \defi{weighted $N_p$-condition}  if it is normally generated (Definition~\ref{defn:proj normal}),
and $F_i$ is generated in degree $\leq w^{i+1}$ for all $i=1,2, \dots, p$ (i.e. the complex $[F_0 \from \cdots \from F_p]$ is Koszul 1-linear, in the sense of Definition~\ref{defn:koszul}).
\end{defn}

In the standard grading, we have $w^{i+1} = i+1$, so our definition extends Green's. E.g., if the variables of $S=k[x_1,x_2,x_3]$ have degrees $1,2$ and $5$; then $w^1=5, w^2=7$, and $w^3=8$.

\medskip

We now turn to our main results.  We establish the following standing hypotheses:

\begin{setup}\label{setup:standard}
Let $C$ be a smooth curve of genus $g$, $L$ a line bundle on $C$, $D$ an effective divisor on $C$, and $d\geq 2$.  Assume $\deg(L \otimes \OO(-D)) \ge 2g + 1$. Let $W$ be the log complete series of type $(D,d)$ for $L$ (see Definition~\ref{defn:log complete}), $S = k[x_0, \dots, x_n]$ the (nonstandard $\Z$-graded) coordinate ring of $\PP(W)$, and $I_C \subseteq S$ the defining ideal of the induced embedding $C\subseteq \PP(W)$. 
\end{setup}

Our first result is a weighted generalization of Castelnuovo's Theorem, i.e. the $p=0$ case of Green's Theorem; see also~\cite{GL85, mattuck, mumford-quadrics}: 
\begin{thm}\label{thm:proj norm}
Under Setup~\ref{setup:standard}, the log complete series $W$ is normally generated.
\end{thm}

The key to Theorem~\ref{thm:proj norm} is using a suitable generalization of the notion of a complete linear series (Definition~\ref{defn:log complete}), as many embeddings of curves into weighted projective spaces simply fail to enjoy any reasonable analogue of normal generation; see Example~\ref{ex:not normally generated}. 

The following generalization of Green's Theorem (Theorem~\ref{thm:green}) is our main result:

\begin{thm}\label{thm:main}
With Setup~\ref{setup:standard}: if $\deg(L \otimes \OO(-D))  \geq 2g+1+q$ for $q \ge 0$, then $C\subseteq \PP(W)$ satisfies the weighted $N_{q+d \cdot \deg(D)}$ condition.
\end{thm}

The theorem shows that, even for embeddings into weighted projective spaces, geometric positivity continues to find expression via low degree syzygies, and in a manner that grows uniformly with $\deg(L)$.  In other words, Green's fundamental insight from Theorem~\ref{thm:green} continues to hold for embeddings into weighted projective spaces.

In fact, as the weighted setting has several distinct notions of ``linearity'', the result even helps bring Green's result into sharper focus, clarifying that positivity is linked with linearity as defined in relation to the Koszul complex, as opposed to alternate notions of linearity, which are equivalent in the standard grading setting, but not in the weighted setting.  In somewhat rough terms, Theorem~\ref{thm:main} says that,
as $q\to \infty$, the Betti table of the (weighted) homogeneous coordinate ring of $C$ ``looks increasingly like the Koszul complex''. The ring $S$ is standard graded if and only if $D = 0$, in which case Theorem~\ref{thm:main} recovers Green's Theorem. The reader may find it helpful to look ahead to \S\ref{sec:betti and np}, where we discuss several detailed examples.
  
As an immediate consequence of Theorem~\ref{thm:main}, we obtain a generalization of the aforementioned theorem of Fujita and Saint Donat stating that a curve embedded by a complete linear series of degree $\geq 2g+2$ is cut out by quadrics. 
It is too much to hope that $C$ will be cut out by
$k$-linear combinations of products $x_ix_j$ (see Example~\ref{ex:P1 deg2}), but this intuition points towards the correct degree bound on the relations:

\begin{cor}\label{cor:quadrics}
With Setup~\ref{setup:standard}: if $\deg(L \otimes \OO(-D))\geq 2g+2$, then $I_C$ is defined by equations of degree at most $2d=\max_{i\ne j} \{ \deg(x_ix_j)\}$.
\end{cor}
In other words, Corollary~\ref{cor:quadrics} implies that the degrees of the defining equations of $C$ are bounded by the maximal degree of a syzygy of $\mathfrak m$. 
A number of results in the literature have a similar form to Corollary~\ref{cor:quadrics}, showing that certain relations are generated in degree at most twice the degree of one of the $k$-algebra generators, e.g.~\cite{lrz-curves,lrz-surfaces,symonds,stackycurve}.

\medskip

Our proof of Theorem \ref{thm:main} relies on a far more general result relating geometric positivity 
to low degree syzygies of the section ring $R\coloneqq\bigoplus_{e \geq 0} H^0(C,L^e)$:

\begin{thm}\label{thm:virtualNp}
Let $C$ be a smooth curve of genus $g$, $L$ a line bundle on $C$, and $f\colon C \to \PP(W)$ a closed immersion induced by a weighted series $W$  associated to $L$ (see \S\ref{subsec:weighted series} for the definition of a weighted series).\footnote{We do not assume that $W$ is a log complete series.}
Assume $\deg(L)\geq 2g+1$  and that $\dim S_1 > g$.  
Let $F$ be the minimal free resolution of the section ring $R$ over the coordinate ring $S$ of $\PP(W)$.  The generators of each $F_i$ lie in degree $\leq w^{i+1}$ for all $i \leq  \dim W - g - 2$. 
\end{thm}
Theorem~\ref{thm:virtualNp} highlights that the connection between positivity of an embedding and low degree syzygies is quite robust, as it applies to many situations where normal generation fails. 
For instance, if we specialize to ordinary projective space, Theorem~\ref{thm:virtualNp} may be applied to obtain low degree syzygies even in cases where a curve is embedded by an incomplete linear series.  In fact, many of Green's results allow for incomplete linear series, and in the case of an embedding into a standard projective space by an incomplete series, Theorem~\ref{thm:virtualNp} follows from Green's Vanishing Theorem~\cite[Theorem~3.a.1]{green2}.  In the weighted projective case, 
Theorem~\ref{thm:virtualNp} can be applied quite broadly, as it does not involve a log complete hypothesis.

There is an important distinction between Theorem~\ref{thm:virtualNp} and Theorem~\ref{thm:main}: the degree bounds hold for the section ring $R$ and not for the coordinate ring $S/I_C$, respectively. For this reason, Theorem~\ref{thm:virtualNp} yields something like virtual $N_p$-conditions---where we use ``virtual'' in the sense of the theory of virtual resolutions introduced in~\cite{BES}---as $F$ is a virtual resolution of the structure sheaf $\cO_C$.  
Theorem~\ref{thm:virtualNp} shows that the connection between geometric positivity and low degree syzygies, which was first illuminated by Green, holds in tremendous generality as long as one considers {\em virtual} syzygies.

The main theme underlying the technical heart of this paper is the way that, when one passes from a standard to a nonstandard grading, notions of  linearity split apart and yet remain subtly intertwined. More specifically, each of the three weighted notions of linearity of free complexes mentioned above, and discussed in detail in \S\ref{sec:koszul linearity}, come into play in the following ways:
\begin{itemize}
	\item Koszul linearity is closely linked to geometric positivity; that is, it is the right notion for weighted $N_p$-conditions.
		\item Strong linearity---specifically the Multigraded Linear Syzygy Theorem of \cite{linear}--- is essential to our proof of our key technical result Theorem~\ref{thm:virtualNp}.\footnote{In this way, the proofs of our main results echo the proof of Green's Theorem~\ref{thm:green} via Green's Linear Syzygy Theorem, as in~\cite[\S8]{eisenbud}.  We expect that one could also use an analogue of the $M_L$-bundle approach from~\cite{GL} to obtain similar results, though such an approach is complicated by the fact that the weighted series $W$ is generated in distinct degrees.  See \S\ref{subsec:ML}.}
	\item Weighted regularity plays a crucial role in our proof of Theorem~\ref{thm:proj norm}.
\end{itemize}

\medskip
In summary, our main results build on Mark Green's insight that geometric positivity is expressed algebraically in terms of low degree syzygies, though this requires novel viewpoints on nearly all of the objects involved.  Our results provide a proof-of-concept for the broader idea that the ``geometry of syzygies'' literature has analogues in the weighted projective setting, and more generally for embeddings into toric varieties or beyond, bolstering the nascent homological theories for multigraded polynomial rings.  Our work also raises a host of new questions: what might play the role of ``scrolls'' in a weighted projective setting?  Is there a weighted analogue of Green's Conjecture (perhaps for stacky curves)?  See \S\ref{sec:questions} for an array of such questions related to the results in this paper.

\medskip

Let us now give an overview of the paper. We begin in \S\ref{sec:betti and np} with a host of examples illustrating our main results. In \S\ref{sec:embeddings}, we begin a detailed investigation of closed immersions into weighted projective spaces; we introduce in this section our notion of ``log complete series" and prove a number of foundational results.  \S\ref{sec:koszul linearity} contains a detailed discussion of the various weighted flavors of linear free complexes discussed above.  In~\S\ref{sec:technical}, we prove Theorem~\ref{thm:virtualNp}, the central technical result of the paper. In \S\ref{sec:proof of Np results}, we prove the rest of our main results.
 Finally, in \S\ref{sec:questions}, we outline some follow-up questions raised by this work.

\subsection*{Acknowledgments}
We thank David Eisenbud, Jordan Ellenberg, Matthew Satriano, Hal Schenck, Frank-Olaf Schreyer, and Greg Smith for valuable conversations. We also thank the referee for many helpful comments. 

\subsection*{Notation}
Throughout the paper, $k$ denotes a field, and the word ``variety" means ``integral scheme that is separated and finite type over $k$". Given a vector $\d = (d_0, \dots, d_n)$ of positive integers, we let $\PP(\d)$ denote the associated weighted projective space. We always assume $d_0\leq d_1 \leq \cdots \leq d_n$. We often use exponents to indicate the number of weights of a particular degree; for instance, we will write $\PP(1^3,2^2)$ for $\PP(1,1,1,2,2)$. Given a weighted projective space $\PP(\d)$, we always denote its coordinate ring by $S$. That is, $S$ is the $\Z$-graded ring $k[x_0, \dots, x_n]$ with $\deg(x_i) = d_i$. Alternatively, given a weighted vector space $W$, we write $\PP(W)$ for the corresponding weighted projective space with coordinate ring $S=\Sym(W)$.  We write $\m$ for the homogeneous maximal ideal of $S$.

\section{Examples}\label{sec:betti and np}

Before diving into the heart of the paper, we illustrate our main results with some examples. In particular, this section is intended to answer the question:  ``What does a Betti table that satisfies the weighted $N_p$-condition look like?''  
All computations in \verb|Macaulay2|~\cite{M2} of Betti tables throughout this section were performed in characteristic $0$.

\medskip

Theorem~\ref{thm:main} reduces to Green's Theorem (Theorem~\ref{thm:green}) when $\deg(D)=0$. The simplest new cases are therefore when $\deg(D)=1$ and $d=2$, and so we begin with such examples.  

\begin{example}\label{ex:P1 deg2}
Let $C=\PP^1$, $D=[0:1]$ and  $d=2$.  If $L=\cO_{\PP^1}(2)$, then the corresponding log complete series $W$ (see Definition~\ref{defn:log complete}) is $\langle s^2,st,st^3,t^4\rangle$.  This induces a closed immersion
\[
\P^1 \to \P(1^2,2^2) \text{  given by } [s:t] \mapsto [s^2:st:st^3:t^4].
\]
The defining ideal $I_C$ for the curve is generated by the $2\times 2$ minors of
\[
\footnotesize
\begin{pmatrix}
x_0&x_1^2&x_2\\
x_1&x_2&x_3
\end{pmatrix},
\normalsize
\] 
where $\deg(x_0)=\deg(x_1)=1$ and $\deg(x_2)=\deg(x_3)=2$.  Corollary \ref{cor:quadrics} implies that $I$ is generated in degree at most $4$, and we can see that this holds and is sharp.  
A direct computation confirms that $C\subseteq \PP(W)$ is normally generated (see Definition \ref{defn:proj normal}) and that $S/I_C$ is Cohen-Macaulay, as predicted by Theorem~\ref{thm:proj norm}. Theorem~\ref{thm:main} implies that this embedding satisfies the $N_2$-condition, and we can check this by observing the free resolution
\[
F = \left[ S \gets S(-3)^2 \oplus S(-4) \gets S(-5)^2 \gets 0 \right]
\]
of $S/I_C$. Indeed, $F_1$ is generated in degrees $\leq 4=w^2$, and $F_2$ is  generated in degrees $\leq 5=w^3$ and these bounds are sharp.
\end{example}

\begin{example}\label{ex:rational deg 8}
Let us continue with the setup of the previous example, but now take $L=\cO_{\PP^1}(8)$.  The corresponding log complete series $W$ is spanned by $s^8, s^{7}t, \dots,st^{7}, st^{15}$ and $t^{16}$.  This weighted series induces a map $\P^1 \to \PP(W) = \P(1^8,2^2)$. In this case, the defining ideal $I_C$ is given by the $2\times 2$ minors of the matrix
\[
\begin{pmatrix}
x_0&x_1&x_2&x_3&x_4&x_5&x_6&x_7^2&x_8\\
x_1&x_2&x_3&x_4&x_5&x_6&x_7&x_8&x_9\\
\end{pmatrix}.
\]
This ideal is generated in degree at most 4, as predicted by Corollary \ref{cor:quadrics}. Once again, one can directly check that the embedding is normally generated and that $S/I_C$ is Cohen-Macaulay. Theorem~\ref{thm:main} implies that this embedding satisfies the $N_8$-condition, and one verifies this by inspecting
the Betti table\footnote{We follow standard \emph{Macaulay2} formatting of Betti tables, where the entry in the $i^{\th}$ column and the $j^{\th}$ row corresponds to $\dim \Tor_i(S/I,k)_{i+j}$, and a dot indicates an entry of $0$.} of $S/I_C$:
\[
\footnotesize
\begin{matrix}
          & 0 & 1 & 2 & 3 & 4 & 5 & 6 & 7 & 8\\
       0: & 1 & . & . & . & . & . & . & . & .\\
       1: & . & 21 & 70 & 105 & 84 & 35 & 6 & . & .\\
       2: & . & 14 & 84 & 210 & 280 & 210 & 84 & 14 & .\\
       3: & . & 1 & 14 & 63 & 140 & 175 & 126 & 49 & 8
       \end{matrix}
  \normalsize
\]
Indeed, $F_i$ is generated in degree $\leq 3+i=w^{i+1}$ for $1\leq i \leq 8$; once again, the bounds imposed by the weighted $N_p$-conditions are sharp.
\end{example}

\begin{example}\label{ex:genus 2}
Let $C$ be the genus $2$ curve defined by the equation 
$z_{2}^{2}-z_{1}^{6}-5z_{0}^{3}z_{1}^{3}-z_{0}^{6}$ in $\mathbb P(1,1,3)$.
Suppose $d=2$ and $D$ is the single point $[1:0:1]$.
Let $L$ be a line bundle of degree 10 on $C$.
We have
\[
 \deg(L \otimes \OO(-D))=9=2g+1+q, \qquad \text{with $q=4$.}
 \]
 It follows from Riemann-Roch that $\PP(W) = \PP(1^8, 2^2)$, and so the associated log complete series induces an embedding $C\subseteq \PP(1^8, 2^2)$.  Theorem~\ref{thm:proj norm} shows that $C\subseteq \PP(1^8,2^2)$ is normally generated and that its coordinate ring is Cohen-Macaulay, which does not seem obvious (at least to these authors).
 By Theorem~\ref{thm:main}, this embedding satisfies the $N_{q+d\cdot \deg D}=N_6$-condition.  A computation in \verb|Macaulay2| yields the following Betti table for $S/I_C$:
\[
\footnotesize
\begin{matrix}
          & 0 & 1 & 2 & 3 & 4 & 5 & 6 & 7 & 8\\
       0: & 1 & . & . & . & . & . & . & . & .\\
       1: & . & 19 & 58 & 75 & 44 & 5 & . & . & .\\
       2: & . & 14 & 80 & 186 & 220 & 136 & 26 & 2 & .\\
       3: & . & 1 & 14 & 61 & 128 & 145 & 98 & 23 & 2\\
       4: & . & . & . & . & . & . & . & 6 & 2
       \end{matrix}
\normalsize
\]
Since $w^8 = 10$, we see that the $7^{\th}$ syzygies require a generator of degree $>w^8$, and thus $C$ satisfies the weighted $N_6$ condition, but not the $N_7$ condition.\footnote{Those familiar with the Green-Lazarsfeld Gonality Conjecture~\cite{GL,ein-lazarsfeld-gonality} might  wonder if one can ``see'' the gonality of $C$ in this Betti table.  The answer is ``yes'', but for a trivial reason: see \S\ref{sec:questions}.}
\end{example}

The examples so far have been in the case where $d=2$ and $\deg D=1$.  We now consider the shape of the Betti table in a slightly different setting:

\begin{example}
\label{ex:genus g}
Let us consider the case of a genus $g$ curve $C$, where $d=2$ but now $\deg D=2$.  Assume that $\deg(L \otimes \OO(-D))= 2g+1+q$ for some $q \ge 0$.  By Riemann-Roch, we have $\dim W_1 = g+q+2$ and $\dim W_2=4$, and so $C \subseteq \PP(1^{{g+q+2}},2^4)$.
Theorem~\ref{thm:main} implies that the minimal free resolution of $S/I_C$ satisfies the $N_{q+4}$-condition.  Since there are now $4$ variables of degree $>1$, the shape of the Betti table is more complicated.  Specifically, we have $w^2=4, w^3 = 6, w^4=8$, and $w^{i+1}=w^{i}+1$ for $i \ge 4$.  This implies that the Betti table of the curve has the following shape, where a symbol $*$ indicates a potentially nonzero entry:
\[
\footnotesize
\begin{matrix}
          & 0 & 1 & 2 & 3 & 4 & 5 &\dots &q+4&q+5&\dots&q+g+4 \\
       0: & * & . & . & . & . & . & \dots &.&.&\dots&.\\
       1: & . & * & * & * & * & * & \dots &*&*&\dots &*\\
       2: & . & * & * & * & * & * & \dots &*&*&\dots &*\\
       3: & . & * & * & * & * & * & \dots &*&*&\dots &*\\
       4: & . & .& * & * & * & * & \dots &*&*&\dots &*\\
       5: & . & .& . & * & * & * & \dots &*&*&\dots &*\\
       6: & . & . & . & . & . & . & \dots & .&*&\dots &*
\end{matrix}
\]
      \normalsize
By combining Remark~\ref{rmk:rows} and Corollary~\ref{cor:reg SI}, we see that the $k^{\th}$ row of the Betti table vanishes for $k > 6$.
The key moment for the weighted $N_{q+4}$-condition is in column $q+5$, where Theorem~\ref{thm:main} no longer guarantees that $F_{q+5}$ is generated in degree $\leq w^{q+6} =q+10$.
\end{example}

\begin{remark}\label{rmk:not CM}
For $C=\P^1$, the following minor variants of Example~\ref{ex:P1 deg2} both lead to non-Cohen-Macaulay examples:  $W= \langle s^2,st,st^3,t^4,t^6\rangle$ and $W' = \langle s^2,st,st^5,t^6 \rangle$.  This underscores the challenge in finding an appropriate weighted analogue of complete linear series.
\end{remark}

\section{Closed immersions into weighted projective spaces}
\label{sec:embeddings}
We now begin to lay the technical foundation for this paper, starting with a study of closed immersions into weighted projective space.
\subsection{Weighted series}\label{subsec:weighted series}

Let $Z$ be a variety and $L$ a line bundle on $Z$. A \defi{weighted series} is a finite dimensional, $\Z$-graded $k$-subspace $W \subseteq \bigoplus_{i \in \Z} H^0(Z,L^i)$. Choosing a basis $s_0, \dots, s_n$ of $W$, where $s_i \in W_{d_i} \subseteq H^0(Z, L^{d_i})$, induces a rational map $\phi_W\colon Z \dasharrow \PP(\d)$ in exactly the same way as in the case of ordinary projective space. When the intersection of the zero loci of the $s_i$ is empty, $\phi_W$ is a well-defined morphism. Let $S = k[x_0, \dots, x_n]$ be the $\Z$-graded coordinate ring of $\PP(\d)$. We will only be interested in the case where $\phi_W$ is a closed immersion; we describe sufficient conditions for this in Proposition~\ref{prop:closed embedding} below. In this case, let $I_Z \subseteq S$ denote the homogeneous prime ideal corresponding to the embedding of $Z$ in $\PP(\d)$. The \defi{homogeneous coordinate ring of $\phi_W$} is the $\Z$-graded ring $S / I_Z$.

\begin{remark}
\label{rem:path1}
Before embarking on the results in this section, we 
highlight
some key differences between the behavior of sheaves on weighted and ordinary projective spaces. 
\begin{enumerate}
\item We have $\Pic(\PP(d)) = \{\OO_{\PP(\d)}(\ell m)\}_{\ell \in \Z}$, where $m = \on{lcm}(d_0, \dots, d_n)$ \cite[Theorem 7.1(c)]{BR}; in particular, not every sheaf $\OO_{\PP(\d)}(i)$ is a line bundle. 
\item It can happen that $\OO_{\PP(\d)}(i) \otimes \OO_{\PP(\d)}(j) \ncong \OO_{\PP(\d)}(i + j)$; see e.g. \cite[pp. 134]{BR}. However, it follows from \cite[Corollary 4A.5(b)]{BR} that $\OO_{\PP(\d)}(im) \otimes \OO_{\PP(\d)}(j) \cong \OO_{\PP(\d)}(im + j)$ for all $i, j \in \Z$. 
\item  Given a graded $S$-module $M$, it is not always the case that $\widetilde{M}(j) \coloneqq \widetilde{M} \otimes \OO_{\PP(\d)}(j)$ coincides with $\widetilde{M(j)}$. Indeed, taking $M = S(i)$, this follows from (2). 
\item Not every morphism $Z \to \PP(\d)$ arises as $\phi_W$ for some weighted series $W$. For instance, take $\d = (1,1,2)$. By (1), every line bundle on $\PP(\d)$ is of the form $\OO_{\PP(\d)}(2\ell)$ for some $\ell \in \Z$. In particular, $\OO_{\PP(\d)}(1)$ is not a line bundle, and so there is no line bundle that can induce the map $\PP(\d) \xra{\id} \PP(\d)$. 
\end{enumerate}
\end{remark}

\begin{prop}\label{prop:closed embedding}
Let $W$ be a weighted series with basis $s_i \in H^0(Z,L^{d_i})$ for $0 \le i \le n$.
Assume there exists $\ell > 0$ such that the map $S_{\ell} \to H^0(Z, L^{\ell})$ induced by $\phi_W : Z \to \PP(\d)$ is surjective and that $L^{\ell}$ is very ample. The morphism $\phi_W$ is a closed immersion.
\end{prop}

\begin{proof}
Let $f_0, \dots, f_r$ be a basis of $H^0(Z,L^{\ell})$.
For $0 \le i \le r$, choose $F_i \in S_{\ell }$ such that $F_i(s_0, \dots, s_n)=f_i$.  The linear series determined by the $F_i$ induces a rational map $\psi: \PP(\d) \dashrightarrow \PP^r$. Write $U$ for the domain of definition of $\psi$.  The image of $\phi_W$ lands in $U$ since the restriction of the $F_i$ to $Z$ is $f_0, \dots, f_r$, which is base-point free. By construction, the composition $\psi \circ \phi_W$ is the morphism induced by $|L^{\ell }|$. Since $L^{\ell }$ is very ample, the map $\psi \circ \phi_W$ is a closed immersion, and so $\phi_W$ is a closed immersion into $U$~\cite[0AGC]{stacks}; since $Z$ is proper, it follows that $\phi_W$ is a closed immersion as well.
\end{proof}

\begin{remark}\label{rmk:not pathologies on stack}
The pathologies described in Remark \ref{rem:path1} all disappear when one works over the associated weighted projective \emph{stack}; see~\cite[\S7]{GS} or \cite[Theorem 2.6]{perroni}. However, the stack introduces its own 
complexities.
For instance, the proof of Proposition~\ref{prop:closed embedding} fails: letting $\PP_{\text{stack}}(W)$ denote the associated stack and defining $\tau_W\colon Z \to \PP_{\text{stack}}(W)$ in the same way as $\phi_W$, the composition
$Z\xra{\tau_W} \PP_{\text{stack}}(W)\dashrightarrow \PP^r$ being a closed immersion does not imply that $\tau_W$ is a closed immersion.  For a simple counterexample, one can let $Z$ be a point and $\tau_W$ be any map to a stacky point.

\end{remark}

\begin{example}
\label{closedex}
Let $Z=\PP^1$ and $L=\cO_{\PP^1}(2)$. The weighted series $W$ spanned by $s^2, st \in H^0(Z, L)$ and $st^3, t^4 \in H^0(Z, L^2)$ induces a map $\phi_W\colon \P^1 \to \P(1,1,2,2)$ given by $[s:t] \mapsto [s^2:st:st^3:t^4]$  (this is the map in Example~\ref{ex:P1 deg2}). Applying Proposition \ref{prop:closed embedding} with $\ell = 2$ implies that $\phi_W$ is a closed immersion. Indeed, we have a commutative diagram
\[
\xymatrix{
\PP^1 \ar[rr]^-{|\cO(4)|}\ar[rrd]_-{\phi_W} && \PP^4\\
&& \PP(1,1,2,2),\ar[u]
}
\]
where the vertical arrow is given by $[x_0:x_1:x_2:x_3]\mapsto [x_0^2:x_0x_1:x_1^2:x_2:x_3]$.
\end{example}
\begin{prop}
\label{stackyprop}
Let $W$ be a weighted series with basis $s_i \in H^0(Z,L^{d_i})$ for $0 \le i \le n$ and $f\colon S \to \bigoplus_{i \ge 0} H^0(Z, L^i)$ the ring homomorphism given by $x_i \mapsto s_i$. If $\phi_W$ is a closed immersion, then $I_Z = \ker(f)$. 
\end{prop}

\begin{proof}
The ideal $I_Z$ is the unique homogeneous prime ideal $P \subseteq S$ such that $Z= V(P)$, where $V( - )$ is the assignment that sends any homogeneous ideal in $S$ to its associated subvariety in $\PP(\d)$ (see e.g. \cite[\S 5.2]{CLS}). Since $\bigoplus_{i \ge 0} H^0(Z, L^i)$ is a domain, $\ker(f)$ is prime, and clearly $V(\ker(f)) = Z$.
\end{proof}

\begin{defn}\label{defn:degenerate}
A closed immersion $Z \subseteq \PP(\mathbf d)$ is \defi{nondegenerate} (resp. \defi{degenerate}) if the defining ideal $I_Z$  is (resp. is not) contained in $\mathfrak m^2$.
\end{defn}
Let $Z\subseteq \PP(\mathbf d)$ be a closed immersion with defining ideal $I_Z\subseteq S$.  Since any basis of $\mathfrak m/\mathfrak m^2$ can be lifted to give algebra generators for $S$, the immersion $Z \subseteq \PP(\d)$ is degenerate if and only if $I_Z$ contains an element in some minimal generating set for $\mathfrak m$.  For instance, by examining their defining ideals, one can see that the curves in Examples \ref{ex:P1 deg2} and \ref{ex:rational deg 8} are both nondegenerate.

\begin{example}
\label{ex:naive}
Suppose $Z=\PP^1$, $L = \OO(1)$, and $W = H^0(\PP^1, \OO(1)) \oplus H^0(\PP^1, \OO(2))$. This is a fairly naive way to generalize a complete linear series, as we have simply taken all sections of degrees $1$ and $2$.  The map $\phi_W\colon \PP^1 \to  \PP(1^2,2^3)$ given by $[s:t]\mapsto [s:t:s^2:st:t^2]$ is a closed immersion, by Proposition~\ref{prop:closed embedding}.
In this case, $Z \subseteq \PP(1^2, 2^3)$ is degenerate since $I_Z$ contains $z_2-z_0^2$, $z_3-z_0z_1$, and $z_4-z_1^2$.
\end{example}

\subsection{Log complete series}\label{sec:complete linear}
We now ask: what is a weighted projective analogue of a \emph{complete} linear series? Before we state our proposed definition, we fix the following notation: given a divisor $L$ and an effective divisor $D$ on a variety $Z$, we write $H^0(Z,L)_D$ for the subspace of sections that vanish along $D$.

\begin{defn}\label{defn:log complete}
Let $Z$ be a smooth projective variety, $L$ a line bundle on $Z$, $D$ an effective divisor on $Z$, and $d\geq 2$.  A weighted series $W$ is a \defi{log complete series of type $(D,d)$} if $W_1 = H^0(Z,L)_D$, $ H^0(Z,L^d) = W_d \oplus \im(\Sym_d(W_1) \to H^0(Z, L^d))$,
and $W_i=0$ for $i\ne 1,d$.
\end{defn}
\begin{example}\label{ex:P1 deg2 reboot}
Let us revisit Example~\ref{ex:P1 deg2}, where $C=\PP^1$, $D=[0:1]$ and  $d=2$.  In this case, $W_1 = \langle s^2:st\rangle$ are the sections of $L$ vanishing at $D$; and $W_2 = \langle st^3:t^4\rangle$.
\end{example}

\begin{remark}
While log complete series provide a strong analogue of complete linear series in the weighted setting, we do not claim that this is a comprehensive analogue.  In fact, one could easily imagine minor variants of our setup that would be generated in 3 or more distinct degrees.  As with analogues of linear resolutions, we expect that there are distinct analogues of complete linear series that lead in different directions.  We restrict attention to log complete series because they strike a good balance.  On one hand, they are sufficiently rich to allow for a wide range of new applications and for our overarching goal of investigating the extent to which Green's results hold in nonstandard graded settings.  On the other hand, they yield embeddings into fairly simple weighted projective spaces of the form $\PP(1^a,d^b)$, thus avoiding some of the pathologies of arbitrary weighted spaces.
\end{remark}

When $D = 0$, Definition~\ref{defn:log complete} recovers the usual notion of a complete linear series. A log complete series $W$ of type $(D, d)$ is unique up to isomorphism of graded vector spaces; we will therefore refer to \emph{the} log complete series of type $(D, d)$. Observe that, when $L^d$ is base-point free, the intersection of the zero loci of the sections in $W$ is empty, so that $W$ induces a well-defined morphism $\phi_W\colon Z \to \PP(\d)$.

\begin{lemma}
\label{lem:H0iso}
Let $Z$ be a curve, $L$ a line bundle on $Z$, and $D$ an effective divisor on $Z$. There is an isomorphism $H^0(Z, L^a)_{bD} \cong H^0(Z, L^a \otimes \OO(-bD))$ for all $a, b \in \Z$.
\end{lemma}

\begin{proof}
Since $\OO_{bD}$ is the structure sheaf of a 0-dimensional scheme, $L \otimes \OO_{bD} \cong L$. Twisting the short exact sequence $0  \to \OO(-bD) \to \OO \to \OO_{bD} \to 0$ by $L$, we therefore arrive at the short exact sequence
$
0\to  L^a \otimes \OO(-bD)\to L^a \to  \OO_{bD} \to 0.
$
The long exact sequence in cohomology yields
$
 H^0(Z, L^a)_{bD} \coloneqq  \ker (H^0(Z, L^a) \to H^0(Z, \OO_{bD})) \cong H^0(Z,L^a \otimes \OO(-bD) ).
 $
\end{proof}

\begin{prop}\label{cor:v ample immersion}
Let $Z,L, D$, and $d$ be as in Definition~\ref{defn:log complete} and $W$ the log complete series of type $(D, d)$. Assume that $L^d$ is very ample.
\begin{enumerate}
\item The canonical map $S_{d} \to H^0(Z,L^d)$ is surjective.
\item The induced map $\phi_W : Z \to \PP(\d)$ is a nondegenerate closed immersion.
\item $W$ is maximal, in the following sense: any weighted series concentrated in degrees 1 and $d$ that properly contains $W$ is degenerate. 
\end{enumerate}
\end{prop}

\begin{proof}
Part (1) follows from the definition of a log complete series. We may apply Proposition~\ref{prop:closed embedding}, with $\ell = d$, to conclude that $\phi_W$ is a closed immersion.
Nondegeneracy holds since the image of the map $W_1^{\otimes d} \to H^0(Z, L^d)$ intersects $W_d$ trivially. This proves (2). If we were to add a section of $\Sym_d(W_1)$ to $W_d$, then it would be in the image of the map $W_1^{\otimes d} \to H^0(Z, L^d)$, forcing degeneracy. Similarly, if we were to add a section $s$ of $H^0(Z, L)$ to $W_1$, the image of $W_1^{\otimes d} \to H^0(Z, L^d)$ would intersect $W_d$ nontrivially; this gives (3). 
\end{proof}
For explicit examples of log complete series, see Examples~\ref{ex:P1 deg2} and \ref{ex:rational deg 8} above.

\subsection{A weighted analogue of normal generation}
\label{sec:normality}
Classically, a closed immersion of a variety $Z$ in $\PP^n$ is projectively normal if the coordinate ring $S/I_Z$ of the immersion is integrally closed \cite[Ex I.3.18]{hartshorne}. If $Z$ is normal, and the closed immersion is induced by the line bundle $L$, then the integral closure of $S/I_Z$ is the section ring $\bigoplus_{i \in \Z} H^0(Z, L^i)$, and so the immersion is projectively normal if and only if the canonical map $S \to \bigoplus_{i \in \Z} H^0(Z, L^i)$ is surjective. In this case, the line bundle $L$ is said to be \defi{normally generated} \cite{mumford-quadrics}. Since we have a short exact sequence
$$
0 \to S / I_Z \into \bigoplus_{i \in \Z} H^0(Z, L^i) \to H^1_{\fm}(S/I_Z) \to 0
$$
of graded $S$-modules, one concludes that $L$ is normally generated if and only if $H^1_{\fm}(S/I_Z) = 0$. With this in mind, we make the following

\begin{defn}\label{defn:proj normal}
Let $Z$ be a variety and $Z\subseteq \PP(\mathbf{d})$ a closed immersion defined by a weighted series $W$. We say $W$ is \defi{normally generated} if 
$H^1_{\mathfrak m} (S/I_Z)=0$.
\end{defn}

\begin{remarks}\label{rem:proj normal}
Let $Z \subseteq \PP(\d)$ be a closed immersion induced by a weighted series $W$. 
\begin{enumerate}
\item The weighted series $W$ is normally generated if and only if the depth of the $S$-module $S/I_Z$ is at least 2. In particular, if $Z$ is a smooth curve,  then $W$ is normally generated if and only if $S/I_Z$ is a Cohen-Macaulay ring.
\item Let $T = S/I_Z$.  Since $I_Z$ is prime, $H^0_{\fm}(T) = 0$, and so, by \cite[Theorem A4.1]{eisenbudbook}, we have a short exact sequence of graded $S$-modules
$$
0 \to T\into \bigoplus_{i \in \Z} H^0(Z, \widetilde{T(i)}) \to H^1_{\fm}(T) \to 0.
$$
Thus, $W$ is normally generated if and only if the canonical map $S_i \to H^0(Z, \widetilde{T(i)})$ is surjective for all $i$, echoing the classical definition.
\end{enumerate}
\end{remarks}

\begin{remark}
Unlike the classical case, normal generation of $W$ is not equivalent to $S/I_Z$ being integrally closed, even when $Z$ is normal. For instance, it follows from Theorem~\ref{thm:proj norm} that the weighted series from Example~\ref{ex:P1 deg2} is normally generated. However, using the notation of that example, the ring $S / I_C  = k[s^2, st, st^3, t^4] \subseteq k[s,t]$ is not integrally closed. Indeed, $t^2 = st^3 / st$ is in the field of fractions of $S/I_C$ but not in $S/I_C$, and it is a root of the polynomial $z^2 - t^4  \in (S/I_C)[z]$. 
\end{remark}

\begin{example}\label{ex:not normally generated}
As mentioned in the introduction, many weighted series fail to be normally generated.  For instance, take a weighted series $W$ such that $W_1$ is a base-point free, incomplete linear series that yields an embedding $Z\to \PP(W_1)$.  We have $H^1_{\mathfrak m}(S/I_Z)_1\ne 0$, and thus $W$ fails to be normally generated.  We thus see that, for very positive linear series to have any hope of normal generation, adding a base locus to $W_1$ is necessary; this observation was a key motivation for our definition of log complete series.
\end{example}

\section{Linearity of free resolutions in the weighted setting}\label{sec:koszul linearity}

There are multiple ways to extend the definition of a linear free resolution to the weighted setting, each with its advantages and disadvantages. We will consider three such notions:
\begin{enumerate}
\item Perhaps the most obvious definition of linearity in the weighted setting is \defi{strong linearity}, which requires all differentials in the resolution to be expressible as $k$-linear combinations of the variables; see Definition~\ref{def:strongly linear} below.  This notion was defined and studied in our paper \cite{linear}, and it is closely related to the multigraded generalization of the BGG correspondence~\cite{hhw}. 
\item We will often find that strong linearity is too restrictive for our purposes. There is a weaker---and more well-known---notion of linearity based on a weighted analogue of Castelnuovo-Mumford regularity and arising from invariant theory~\cite{benson, symonds}, which we call \defi{weighted regularity}. It is determined by the number of rows in the Betti table of the resolution; see Definition~\ref{defn:benson} for details. 
\item Weighted regularity, however, is too weak of a condition for us; we therefore introduce in this paper an intermediate notion between (1) and (2) called \defi{Koszul linearity} (Definition~\ref{defn:koszul}). Roughly speaking, a free resolution is Koszul linear if its Betti numbers grow no faster than those of the Koszul complex. Our definition of weighted $N_p$-conditions (Definition~\ref{defn:1regular}) is based on Koszul linearity.
\end{enumerate}

Each of (1) - (3) will be used in the proofs of our main results. In the standard graded case, each notion gives an alternative but equivalent way to view linear resolutions; Example~\ref{ex:rational deg 4} below illustrates how these notions diverge in the general weighted case. To briefly explain: while weighted regularity only depends on the number of rows in the Betti table of the resolution, Koszul linearity involves more granular information about the Betti numbers. Moreover, strong linearity cannot be detected from the Betti numbers of the resolution at all, as one can see from Example \ref{easyex} below.  See also~\cite{bounds}, which explores the relationship between these notions in greater detail.

\subsection{Strong linearity}
\label{subsec:strong}
Our most restrictive notion of linearity for nonstandard graded free resolutions is the following:
\begin{defn}[\cite{linear} Definition 1.1]\label{def:strongly linear}
A complex $F$ of graded free $S$-modules is \defi{strongly linear} if there exists a choice of basis of $F$ with respect to which its differentials may be represented by matrices whose entries are $k$-linear combinations of the variables.
\end{defn}

In the nonstandard graded setting, strong linearity of a free complex $F$ cannot be detected by the degrees of its generators, as the following simple example illustrates:

\begin{example}
\label{easyex}
Suppose $S=k[x_0,x_1]$, where the variables have degrees $1$ and $2$. Consider the complexes
$S \overset{x_0^2}{\longleftarrow} S(-2)$ and $S \overset{x_1}{\longleftarrow} S(-2)$; only the second complex is strongly linear.
\end{example}

The main goal of our paper \cite{linear} is to establish a theory of linear \emph{strands} of free resolutions in the nonstandard graded context, culminating in a generalization of Green's Linear Syzygy Theorem \cite{green}: that circle of ideas will play a key role in this paper. Before we recall the details, we briefly discuss some background on (a weighted analogue of) the Bernstein-Gel'fand-Gel'fand (BGG) correspondence. We refer the reader to \cite[\S 2.2]{tate} for a detailed introduction to the multigraded BGG correspondence, following work of \cite{hhw}. 
 
\subsubsection{The weighted BGG correspondence} 
\label{sec:BGG}
Let $E = \bigwedge_k(e_0, \dots, e_n)$  be an exterior algebra, equipped with the $\Z^2$-grading given by $\deg(e_i) = (-\deg(x_i); -1)$. Denote by $\Com(S)$ the category of complexes of graded $S$-modules and $\DM(E)$ the category of \emph{differential $E$-modules}, i.e. $\Z^2$-graded $E$-modules $D$ equipped with a degree $(0; -1)$ endomorphism that squares to 0. The weighted BGG correspondence is an adjunction
$$
\LL : \DM(E) \leftrightarrows \Com(S) : \RR
$$
that induces an equivalence on derived categories. We will only be concerned in this paper with the functor $\LL$ applied to $E$-modules: if $N$ is a $\Z^2$-graded $E$-module, the complex $\LL(N)$ has terms and differential given by
$$
\LL(N)_j = \bigoplus_{a \in \Z} S(-a) \otimes_k N_{(a;j)} \quad \text{and} \quad s \otimes n \mapsto \sum_{i = 0}^n x_is \otimes e_in.
$$
The complex $\LL(N)$ is strongly linear, and in fact every strongly linear complex of $\Z$-graded $S$-modules is of the form $\LL(N)$ for some $E$-module $N$ \cite{linear}.

\subsubsection{Strongly linear strands} 

\begin{defn}[\cite{linear}]
\label{stranddef}
Let $M$ be a graded $S$-module such that there exists $a \in \Z$ with $M_a \ne 0$ and $M_{< a} = 0$. We set $E^* = \Hom_k(E, k)$, considered as an $E$-module via contraction. The \defi{strongly linear strand} of the minimal free resolution of $M$ is $\LL(K)$, where $\LL$ is the BGG functor defined above, and 
$$
K = \ker\left( M_a \otimes_k E^*(-a;0) \xra{\sum_{i = 0}^n x_i \otimes e_i} \bigoplus_{i = 0}^n M_{a + d_i} \otimes_k E^*(-a - d_i ; -1)\right).
$$
\end{defn}

In the standard graded case, Definition~\ref{stranddef} recovers the classical notion of the linear strand of a free resolution \cite[Corollary 7.11]{eisenbud}. When $M$ is generated in a single degree, the strongly linear strand of the minimal free resolution $F$ of $M$ may be alternatively defined as follows: it is the unique maximal strongly linear subcomplex $F'$ of $F$ such that $F'$ is a summand (as an $S$-module, but not necessarily as a complex) of $F$ \cite{linear}.

A main result of \cite{linear} is a multigraded generalization of Green's Linear Syzygy Theorem \cite{green}. We  recall the statement of this theorem in the nonstandard $\Z$-graded case:

\begin{thm}[\cite{linear} Theorem 6.2]
\label{MLST'}
Let $M$ be a finitely generated $\Z$-graded $S$-module and $F$ its minimal free resolution. Suppose $M_0 \ne 0$, and $M_i = 0$ for $i < 0$. The length of the strongly linear strand of $F$ is at most $\max\{\dim M_0 - 1, \dim R_0(M)\}$,
where $R_0(M)$ is the variety of rank one linear syzygies of $M$, i.e.
$$
R_0(M) = \{w \otimes m \in W \otimes_\kk M_0 \text{ : } wm = 0 \text{ in } M\}.
$$
\end{thm}

The following geometric consequence of Theorem~\ref{MLST'} plays a crucial role in all of our main results. It extends to weighted projective spaces a result originally proven by Green \cite{green2} over projective space; see also \cite[Corollary 7.4]{eisenbud}. 
\begin{thm}
\label{MLST}
Let $Z$ be a variety, $L$ a line bundle on $Z$, and $W$ a weighted series associated to $L$ such that the associated map $\phi_W : Z \to \PP(\d)$ is a nondegenerate closed embedding. Let $V$ be a vector bundle on $Z$ and $M$ the $S$-module $\bigoplus_{i \in \Z} H^0(Z, V \otimes L^i)$. Assume $M_0 \ne 0$, and $M_i = 0$ for $i < 0$. 
The strongly linear strand of the minimal $S$-free resolution of $M$ has length at most $\dim{M_0} - 1$. 
\end{thm}

While Theorem~\ref{MLST} follows directly from ideas in our paper \cite{linear} (cf. \cite[Corollary 1.5]{linear}), we include a detailed proof.

\begin{proof}
This follows from essentially the same argument as in \cite[Corollary 7.4]{eisenbud} (see also the proof of \cite[Corollary 1.5]{linear}). Let $m \in M_0$ and $w \in W$; recall that $W \subseteq S$ is the $\kk$-vector subspace of $S$ generated by the variables. Notice that $m \otimes w \in R_0(M)$ if and only if $m \otimes w_i \in R_0(M)$ for all homogeneous components $w_i$ of $w$. Assume $m \otimes w \in R_0(M)$ and that $w$ is homogeneous; by Theorem~\ref{MLST'}, it suffices to show that this syzygy is trivial, i.e. either $m=0$ or $w=0$.  Suppose $m\ne 0$, and let $Q$ be a maximal ideal of $S$ such that the image $m_Q$ of $m$ in the localization $M_Q$ is nonzero. Let $I_Z$ be the defining ideal of $Z$ in $\PP(W)$; since $Z$ is integral, $I_Z$ is prime. Let $w_Q$ denote the image of $w$ in $(S/I_Z)_Q$. Notice that $M_Q$ is a free $R_Q$-module, where $R$ is the ring $\bigoplus_{i \in \Z} H^0(Z, L^i)$. Since $R$ is a domain, and the natural map $S/I_Z \to R$ is injective by Proposition~\ref{stackyprop}, the relation $w_Qm_Q = 0$ forces $w_Q = 0$, which implies that $w \in P$. By the nondegeneracy of the embedding, $P$ does not contain a homogeneous linear form; we conclude that $w = 0$.%
\end{proof}

We will need one additional result concerning strongly linear strands: 
\begin{lemma}\label{lem:strongly linear strands}
Let $0\to M'\to M\to M''\to 0$ be a short exact sequence of $S$-modules.  Assume that $M'_a$ and $M_a$ are nonzero, and $M'_i =  M_i = 0$ for $i < a$. Moreover, assume $M''_a = 0$.  There is a natural isomorphism between the strongly linear strands of $M'$ and $M$.
\end{lemma}

\begin{proof}
We assume, without loss, that $a = 0$. Let $L$ be the $\Z^2$-graded $E$-module $\bigoplus_{i = 0}^n M_{ d_i} \otimes_k E^*(-d_i ; -1)$, and define $L'$ and $L''$ similarly. We have a commutative diagram
$$
\xymatrix{
0 \ar[r] & M'_0 \otimes_k E^* \ar[d] \ar[r]^-{\cong} & M_0 \otimes_k E^* \ar[r] \ar[d] & 0 \ar[d] \ar[r] & 0 \\
0 \ar[r] & L'  \ar[r] & L  \ar[r] &  L''   \ar[r] & 0
}
$$
of $\Z^2$-graded $E$-modules, where the rows are exact, and the vertical maps are given by multiplication on the left by $\sum_{i = 0}^n x_i \otimes e_i$. Let $K$ (resp. $K'$) denote the kernel of the middle (resp. leftmost) vertical map. By the Snake Lemma, the natural map $K' \to K$ is an isomorphism, and hence the natural map $\LL(K') \to \LL(K)$ is as well. 
\end{proof}

\subsection{Weighted regularity}
Benson introduced in~\cite{benson} an analogue of Castelnuovo-Mumford regularity for nonstandard $\ZZ$-graded polynomial rings, which we call ``weighted regularity'' to emphasize its connection with weighted projective space:\footnote{This is also a special case of the notion of multigraded regularity defined by Maclagan-Smith \cite{MS}.}

\begin{defn}
\label{defn:benson}
Let $M$ be a finitely generated graded $S$-module. For each $i \ge 0$, set
$$
a_i(M) = \sup\{ j \in \Z \text{ : } H^i_{\m}(M)_j \ne 0\}.
$$ 
The \defi{weighted regularity of $M$} is $\sup\{i \ge 0 \text{ : } a_i(M)+i\}$.
\end{defn}
\begin{remark}\label{rmk:rows}
By a result of Symonds~\cite[Proposition 1.2]{symonds}, if $M$ has weighted regularity $r$, and $F$ is the minimal free resolution of $M$, then $F_j$ is generated in degree at most $r + j + \sum_{i=0}^n (\deg(x_i)-1) $.  Equivalently, the $k^{\th}$ row of the Betti table of any such module vanishes for $k > r + \sum_{i=0}^n (\deg(x_i)-1)$.
\end{remark}

\begin{example}
\label{easyex2}
Let us revisit the two resolutions from Example~\ref{easyex}. 
Recall that $S=k[x_0,x_1]$, where the variables have degrees $1$ and $2$.  Both $S/(x_0^2)$ and $S/(x_1)$ have weighted regularity $0$, and their Betti tables are both
\[
\begin{matrix}
&0&1\\
0&1&.\\
1&.&1
\end{matrix}
\]
In particular, while $S/(x_0^2)$ has weighted regularity $0$, its minimal free resolution is not strongly linear.  By contrast, any module that is generated in degree 0 and has a strongly linear free resolution is weighted 0-regular (see Remarks \ref{rmks:koszul}(1) and Proposition~\ref{prop:strongkoszul} below).
\end{example}

\begin{example}
In Corollary~\ref{cor:reg SI}, we will prove that, under Setup~\ref{setup:standard}, the weighted regularity of $S/I_C$ is $2$ if $g>0$ and $1$  if $g=0$. 
For instance, consider the genus $2$ curve from Example~\ref{ex:genus 2} embedded in $\PP(1^8,2^2)$.  Its coordinate ring has weighted regularity $2$, and so, by Remark~\ref{rmk:rows}, the Betti table has $2 + \sum_{i=0}^9 (\deg(x_i)-1) = 2 = 2+ 2 = 4$ rows.
\end{example}

\subsection{Koszul linearity}
\label{subsec:koszullinearity}
We fix the following 
\begin{notation} 
\label{notation}
Let $w^i$
(resp. $w_i$) be the sum of the $i$ largest (resp. smallest) degrees of
the variables: that is, $w^i \coloneqq \sum_{j=n-i+1}^n d_j$, and $w_i \coloneqq \sum^{i-1}_{j=0} d_j$. 
\end{notation}
If $K=K_0 \gets K_1 \gets \cdots$ is the Koszul complex on $x_0, \dots, x_n$, then $w_i$ is the smallest degree of a generator of $K_i$, and $w^i$ is the largest such degree.  

\begin{defn}
\label{defn:koszul}
A minimal free complex $[F_0 \overset{\phi_1}{\from} F_1 \overset{\phi_2}{\from}F_2   \cdots]$ of graded $S$-modules is \defi{Koszul $a$-linear} if each $F_i$ is generated in degrees $< w^{i+1} + a$; by minimal we mean $\phi_i(F_i)\subseteq \mathfrak mF_{i-1}$. We sometimes abbreviate Koszul $0$-linear to simply ``Koszul linear''.
\end{defn}

\begin{remarks}
\label{rmks:koszul}
We observe the following:
\begin{enumerate}
\item If $M$ is as in Definition~\ref{defn:koszul}, and the free resolution of $F$ is Koszul $a$-linear, then it follows from Remark~\ref{rmk:rows} that $M$ is weighted $a$-regular. The converse is false; see Example~\ref{ex:rational deg 4}.
\item The weighted $N_p$-condition from Definition~\ref{defn:1regular} is equivalent to normal generation of the weighted series and Koszul 1-linearity of the complex $[F_0 \gets \cdots \gets F_p]$.
\end{enumerate}
\end{remarks}

Of course, the Koszul complex on $x_0, \dots, x_n$ is Koszul $0$-linear. More generally, we have:
\begin{prop}
\label{prop:strongkoszul}
Let $M$ be a graded $S$-module that is generated in a single degree $a$. If the minimal free resolution $F$ of $M$ is strongly linear, then it is Koszul $a$-linear. 
\end{prop}
\begin{proof}
Since $F$ is strongly linear and $M$ is generated in a single degree, $F$ is equal to its strongly linear strand $\LL(K)$, where $K$ is as in Definition~\ref{stranddef}. It therefore follows from the definition of $K$ that $F$ is a summand (as an $S$-module, but not as a complex) of a direct sum of copies of $\LL(E^*(-a;0))$. Finally, observe that $\LL(E^*(-a;0))$ is the Koszul complex with $0^{\th}$ term generated in degree $a$; the result immediately follows. 
\end{proof}

\begin{example}
The converse of Proposition \ref{prop:strongkoszul} is false. Returning Example \ref{easyex}, the complex $S \overset{x_0^2}{\longleftarrow} S(-2)$ is Koszul $0$-linear but not strongly linear. 
\end{example}

\begin{example}\label{ex:rational deg 4}
Let $C = \PP^1$, $L = \OO_C(5)$, and $D$ the divisor $[0:1] + [1:0]$. 
The associated log complete seres induces an embedding $\PP^1 \subseteq \PP(1^4,2^4)$ given by 
\[
[s:t]\mapsto [s^{4}t:s^{3}t^{2}:s^{2}t^{3}:st^{4}:s^{9}t:s^{10}:st^{9}:t^{10}].
\]
The Betti table is
\[
\footnotesize
\begin{matrix}
          & 0 & 1 & 2 & 3 & 4 & 5 & 6\\
       \text{total:}
          & 1 & 21 & 70 & 105 & 84 & 35 & 6\\
       0: & 1 & . & . & . & . & . & .\\
       1: & . & 3 & 2 & . & . & . & .\\
       2: & . & 12 & 24 & 12 & . & . & .\\
       3: & . & 6 & 36 & 54 & 24 & . & .\\
       4: & . & . & 8 & 36 & 48 & 20 & .\\
       5: & . & . & . & 3 & 12 & 15 & 6
       \end{matrix}
      \normalsize
 \]
From this Betti table, one can check that this resolution is Koszul $1$-linear. For instance, $F_1$ has generators of degree $<5 = w^{2}+1$, $F_2$ has generators of degree $<7 = w^{3}+1$, and so on.
 
The defining ideal $I_C$ is given by the $2\times 2$ minors of the matrix $\left(\begin{smallmatrix}
       x_{0}&x_{1}&x_{2}&x_{4}&x_{5}&x_{3}^{2}&x_{6}\\
       x_{1}&x_{2}&x_{3}&x_{0}^{2}&x_{4}&x_{6}&x_{7}\\
       \end{smallmatrix}\right)$. It follows that the minimal free resolution of $S/I_C$ is the Eagon-Northcott complex of this matrix.  Since this matrix includes the entries $x_3^2$ and $x_0^2$, this minimal free resolution is {\em not} strongly linear.  Thus, even in the case of a rational curve, strong linearity is too restrictive to capture the linearity of the free resolution of the coordinate ring.

Finally, let us analyze the example from the perspective of weighted regularity. By Remark~\ref{rmks:koszul}(1), $S/I_C$ is $1$-regular; by Remark~\ref{rmk:rows}, this says precisely that the $k^{\th}$ row of the Betti table vanishes for $k > 5$.  Thus, for instance, the weighted regularity computation would imply that $F_1$ is generated in degree at most $6$.  We therefore see that weighted regularity is too weak to fully describe the situation.
\end{example}

\section{Proof of Theorem~\ref{thm:virtualNp}}
\label{sec:technical}

We begin by establishing several technical results. The first is a simple calculation:
\begin{lemma}
\label{lem:easy}
Let $S$ be as in Theorem~\ref{thm:virtualNp} and $M$ be a finitely generated $S$-module. 
Assume that $M_0\ne 0$ but $M_i=0$ for $i<0$.
\begin{enumerate}
\item If the Betti number $\beta_{i, j}(M)$ is nonzero, then $j \ge w_i$ (see Notation~\ref{notation}). 
\item Suppose there is a variable $x_\ell \in S$ that is a non-zero-divisor on $M$. Define 
$$
w_i' = \begin{cases} w_i, & i < \ell; \\ w_{i+1} - \deg(x_\ell), & i \ge \ell. \end{cases} 
$$
If $\beta_{i, j}(M) \ne 0$, then $j \ge w_i'$.
\end{enumerate}
\end{lemma}

\begin{proof}
If $K$ is the Koszul complex on the variables $x_0, \dots, x_n$, then the minimal degree of an element of $\Tor_i(M,k)=H_i(M\otimes_{S} K)$ is $w_i $. This proves (1). For (2), let $F$ denote the minimal $S$-free resolution of $M$. Since $x_\ell$ is a non-zero-divisor on $M$, $F / x_\ell F$ is the minimal $S / (x_\ell)$-free resolution of $M / x_\ell M$. Now apply (1) to the $S / (x_\ell)$-module $M / x_\ell M$. 
\end{proof}

The following lemma is an analogue of a well-known result in the standard graded case and is proven in the same way as its classical counterpart.

\begin{lemma}
\label{lem:technical}
Let $C, L$, $R$, $S$, $W$, and $f\colon C \to \PP(W)$ be as in Theorem \ref{thm:virtualNp}.
\begin{enumerate}
\item The graded $S$-module $R$ has depth 2. In particular, $R$ is a Cohen-Macaulay $S$-module and a maximal Cohen-Macaulay $S / I_C$-module. 
\item Let $\om_R = \bigoplus_{i \in \Z} H^0(C, \om_C \otimes L^i)$, and denote by $|\d|$ the sum of the degrees of the variables in $S$. We have
$
\Ext^{n-1}_S(R, S(-|\mathbf{d}|)) \cong   \om_R.
$
\end{enumerate}
\end{lemma}

\begin{proof}
We observe that the canonical map $R \to \bigoplus_{i \in \Z} H^0(\PP(W), \widetilde{R(i)} )$ is an isomorphism, i.e. $R$ is $\m$-saturated. 
By \cite[Theorem A4.1]{eisenbudbook}, we have an exact sequence
$$
0\to H^0_{\mathfrak m} (R) \to R \xra{\cong} \bigoplus_{i \in \ZZ} H^0(\PP(W), \widetilde{R(i)}) \to H^1_{\mathfrak m} (R)\to 0
$$
and isomorphisms
\begin{equation}
\label{eqn:isos}
H^{j+1}_{\m} (R) \cong \bigoplus_{i \in \Z} H^j(\PP(W), \widetilde{R(i)}) = \bigoplus_{i \in \Z} H^j(C, L^i)
\end{equation}
for $j > 0$. In particular, we have $H^i_{\mathfrak m} (R) = 0$ for $i = 0, 1$; that is, $R$ has depth 2. Part (1) now follows from the observation that $\dim {S / I_C} = 2$. As for (2): given a $\Z$-graded $k$-vector space $V$, let $V^*$ denote its graded dual. We have
$
\Ext^{n-1}_S(R, S(-|\d|))  \cong H^2_\fm(R)^*  \cong \bigoplus_{i \in \Z} H^1(C, L^i)^*\cong \om_R,
$
where the first isomorphism follows from local duality, the second from \eqref{eqn:isos}, and the third from Serre duality. 
\end{proof}

Next, we need the following strengthening of
Theorem~\ref{MLST}:

\begin{lemma}\label{lem:delicate}
Suppose we are in the setting of Theorem~\ref{MLST}, and assume $\dim W_1 > \dim M_0$. Let $F$ be the minimal $S$-free resolution of $M$. Any summand of $F_i$ generated in degree $j$ for some $j < w_{i+1}$ (see Notation~\ref{notation}) lies in the strongly linear strand of $F$. In particular, if $\beta_{i,j}(M)\ne 0$ for some $j < w_{i+1}$, then $i \le \dim M_0 - 1$.
\end{lemma}

In the standard graded case, the first  statement in Lemma~\ref{lem:delicate} is tautological: it says that, if a summand of $F_i$ is generated in degree $i $, then it is in the linear strand. However, in the weighted setting, the strongly linear strand cannot be interpreted in terms of Betti numbers (see, for instance, Example~\ref{easyex}), and so Lemma~\ref{lem:delicate} is not at all obvious in general; indeed, our proof is a bit delicate. 

\begin{proof}[Proof of Lemma~\ref{lem:delicate}]
The second statement follows immediately from the first, by Theorem~\ref{MLST}. 
Let $K$ be the Koszul complex on the variables of $S$. We consider classes in $\Tor^S_*(k, M)$ as homology classes in $K \otimes_S M \cong \bigwedge W \otimes_k M$, and we fix once and for all an embedding $\Tor^S_*(k, M) \into Z(\bigwedge W \otimes_k M)$ of $\Z^2$-graded $k$-vector spaces, where the target denotes the cycles in $\bigwedge W \otimes_k M$. In this proof, we will identify classes in $\Tor^S_*(k, M)$ with cycles in $\bigwedge W \otimes_k M$ via this embedding. We may decompose any element $\sigma \in \bigwedge W \otimes_k M$ as 
$\sum_{i\geq 0} \sigma_i$, where $\sigma_i \in \bigwedge W \otimes_k M_i$. 
Let $W_{>1} = \bigoplus_{i > 1} W_i$, so that $\bigwedge W = \bigwedge W_1 \otimes_k \bigwedge W_{>1}$. We may write any $\sigma \in \bigwedge W \otimes_k M$ as $\sum \alpha \otimes \beta \otimes m_{\alpha,\beta}$, where the sum ranges over all pairs $(\alpha, \beta)$ such that $\a$ is an exterior product of basis elements of $W_1$, and $\beta$ is an exterior product of basis elements of $W_{>1}$; here, each $m_{\a,\b}$ is an element of $M$.  We call each nonzero $\alpha \otimes \beta \otimes m_{\alpha,\beta}$ in this sum a \defi{term} of $\sigma$. It is possible that $\a$ (resp. $\b$) is an empty product of basis elements, in which case $\a$ (resp. $\b$) is $1 \in \bigwedge^0 W_1$ (resp. $1 \in \bigwedge^0 W_{>1}$). Given a nonzero element $\sigma \in \bigwedge W \otimes_k M$, we define
$$
\nu(\sigma) = \max\{ m \text{ : } \text{a term of $\sigma$ lies in } \bigwedge^{\dim W_1 - m}W_1 \otimes_k \bigwedge W_{>1}\otimes_k M  \}.
$$
The function $\nu$ measures the maximal number of degree $1$ elements that do not appear in one of the $\alpha$'s. For instance, if $\nu(\sigma)=0$, then, for every term $\alpha \otimes \beta \otimes m_{\alpha,\beta}$ of $\sigma$, $\alpha$ is the product of all of the degree $1$ variables. Let us now prove the following:

\begin{claim*}
If $\sigma$ is a nonzero class in $\Tor_*^S(k, M)$, then $\nu(\sigma) \ne 0$.
\end{claim*}

Indeed, let $x_i$ be a degree 1 variable, $\overline{W}$ the quotient of $W$ by the span of $x_i$, and $\overline{M}$ the corresponding module $M/(x_i)$ over $\overline{S} = S/(x_i)$.  Since $x_i$ is a regular element on $M$, the surjection
$\bigwedge W\otimes_k M \onto \bigwedge \overline{W} \otimes_k \overline{M}$
induces an isomorphism $\theta : \Tor^S_*(k,M) \xra{\cong} \Tor_*^{\overline{S}}(k,\overline{M})$ on homology. Since $\theta(\sigma) \ne 0$, 
$\nu(\sigma)$ must be nonzero; this proves the claim.

Now, let $\sigma$ be a nonzero class in $\Tor^S_i(k, M)_j$, where $j < w_{i+1}$. It suffices to show that $\sigma = \sigma_0$; this implies that $\sigma$ lies in the strongly linear strand. Assume, toward a contradiction, that $\sigma_\ell \ne 0$ for some $\ell > 0$. Since $\sigma_\ell \in \bigwedge^i W \otimes M_\ell$, we have $w_i + \ell \le j  < w_{i+1}$.  Recalling that $w_{i+1} - w_i = d_{i+1} \coloneqq \deg(x_{i+1})$, this implies $d_{i+1} > \ell \geq 1$.  We conclude that \begin{equation}\label{eqn:pg1}
i \geq \dim W_1.
\end{equation}
There are two cases to consider.
\vskip\baselineskip
\noindent \emph{Case 1: $\nu(\sigma_\ell)>0$ for some $\ell > 0$.}
In this case, $\sigma_\ell$ has some term $\alpha \otimes \beta \otimes m_{\alpha,\beta}$ such that $\alpha$ is not divisible by a degree $1$ variable; without loss of generality, let us say $\alpha$ is not divisible by $x_0$.  It follows that $\deg(\alpha\otimes \beta) \geq \deg(x_1x_2\cdots x_i)=w_{i+1}-1$. Thus,
\[
\deg(\sigma_\ell) = \deg(\alpha\otimes \beta) + \ell \geq w_{i+1}-1+\ell \ge w_{i+1}.
\]
This is impossible, since $\deg(\sigma_\ell) = \deg(\sigma) <w_{i+1}$. 
\vskip\baselineskip
\noindent \emph{Case 2: $\nu(\sigma_\ell) = 0$ for all $\ell > 0$.} 
For every term $\alpha\otimes\beta\otimes m_{\a, \b}$ of $\sigma_\ell$ for $\ell > 0$, we have $\beta \in \bigwedge^{i-\dim W_1}W_{>1}$. On the other hand, it follows from the Claim above that there must be some term $\alpha' \otimes\beta'\otimes m_{\a',\b'}$ of $\sigma_0$ such that $\beta' \in \bigwedge^{i-\dim W_1+t}W_{>1}$ for some $t>0$; recall that, by \eqref{eqn:pg1}, $i - \dim W_1 \ge 0$. Let $E = \bigwedge W^*$, and notice that $\bigwedge W \otimes_k M$ is an $E$-module via the contraction action of $E$ on $\bigwedge W$. We may choose $f \in \bigwedge^{i-\dim W_1+1}W_{>1}^* \subseteq E$ such that $f\sigma_0\ne 0$; notice, however, that $f\sigma_\ell=0$ for all $\ell>0$. Thus, $f\sigma = f\sigma_0=(f\sigma)_0 \in \bigwedge W \otimes M_0$. Moreover, since $\sigma \in \bigwedge W \otimes_k M$ is a cycle, $f\sigma$ is also a cycle, as the Koszul differential on $\bigwedge W \otimes_k M$ is $E$-linear. Thus, since $f\sigma = (f\sigma)_0$, it follows from the definition of the strongly linear strand (Definition~\ref{stranddef}) that $f\sigma$ determines a summand of the strongly linear strand of $F$. But $f\sigma$ has homological degree $i- (i-\dim W_1 + 1)= \dim W_1 -1$, and so $\dim W_1 -1 \leq \dim M_0-1$, by Theorem~\ref{MLST}. This contradicts our assumption that $\dim W_1 > \dim M_0$.
\end{proof}

In the standard graded case, the proof of Green's Theorem (Theorem~\ref{thm:green}) via the Linear Syzygy Theorem (cf. \cite[Theorem 8.8.1]{eisenbud}) makes use 
of numerous statements about linear strands that rely on degree arguments.  These break down in the nonstandard graded situation, and Lemmas ~\ref{lem:easy}---\ref{lem:delicate} act to fill that gap.  Thus, with these lemmas in hand, we can now turn to the proof of  Theorem~\ref{thm:virtualNp}.

\begin{proof}[Proof of Theorem \ref{thm:virtualNp}]
Recall that $d_0, \dots, d_n$ are the degrees of the variables $x_0, \dots, x_n$ in $S$, and we assume $d_0 \le d_1 \le \cdots \le d_n$. As in Lemma~\ref{lem:technical}(2), we let $|\d| = \sum_{i = 0}^n d_i$ and $\om_R=\bigoplus_{i \in \Z} H^0(C,\omega_C \otimes L^i)$. We remark, for later use, that $\dim (\om_R)_0=H^0(C,\omega_C) = g$. By Lemma~\ref{lem:technical}(2), we have
$
\Ext^{n-1}_S(R, S(-|\mathbf{d}|)) \cong \om_R.
$
Letting $F$ be the minimal $S$-free resolution of $R$ and $F^\vee = \Hom_S(F, S)$, it follows that 
$F^\vee(- |\mathbf{d}|)[-n+1]$ 
is the minimal free resolution of $\om_R$. In particular, we have 
$
\beta_{i,j}(R) = \beta_{n - 1 - i,  |\mathbf{d}| - j}(\om_R).
$
Now, suppose $\beta_{i, j}(R) = \beta_{n-1-i,  |\mathbf{d}| - j}(\om_R) \ne 0$, and assume $j > w^{i+1}$. We now compute:
$$
|\mathbf{d}| - w_{n-i} = \sum_{j=0}^n d_j - \sum_{j=0}^{n-1-i} d_j = \sum_{j=n-i}^n d_j = w^{i+1} < j.
$$
Rearranging this inequality, we have $|\d| - j  < w_{n-i}$. There are now two cases to consider.
\vskip\baselineskip
\noindent{\emph{Case 1: $g = 0$}.}  In this case, $(\om_R)_1 \ne 0$, and $(\om_R)_i = 0$ for $i < 1$. Every variable $x_i \in S$ is a non-zero-divisor on $\om_R$. In particular, $x_0$ has this property; recall that $\deg(x_0) = 1$. Applying Lemma~\ref{lem:easy}(2), with $\ell = 0$, we arrive at the inequality $|\d| -j\ge w_{n-i}$, a contradiction. We therefore conclude that, if $\beta_{i, j}(R) \ne 0$, then $j \le w^{i+1}$. 
\vskip\baselineskip
\noindent{\emph{Case 2: $g > 0$}.} We now have $(\om_R)_0 \ne 0$, and $(\om_R)_i = 0$ for $i < 0$. Applying Lemma~\ref{lem:delicate} to $\om_R$ implies that $n - 1 - i < \dim (\om_R)_0 = g$, i.e. $i > n - 1 - g = \dim W - g - 2$. 
\end{proof}

Let us illustrate the proofs of both Theorem~\ref{thm:virtualNp} and Lemma~\ref{lem:delicate} via an example:
\begin{example}
\label{ex:lemma}
Suppose we are in the setting of Theorem~\ref{thm:virtualNp}, and assume $g = 2$ and $\PP(W) = \PP(1^6,2^4)$. Let $\om_R$ be as in Lemma~\ref{lem:technical}(2). To prove Theorem~\ref{thm:virtualNp} in this example, we must show that the columns of the Betti table of $\om_R$ are bounded above by the dots in the diagram below:\footnote{We are using here that the $k^{\th}$ row in the Betti table of $R$ must vanish for $k > 6$. One sees this by combining Remark~\ref{rmk:rows} with the fact that the weighted regularity of $R$ is 2, which we prove in Corollary~\ref{cor:reg SI}.}
\[
\footnotesize
\begin{matrix}
          & 0 & 1 & 2 & 3 & 4 & 5 & 6 & 7 & 8\\
       0: & \bullet & \bullet & . & . & . & . & . & . & .\\
       1: & . &.   &   \bullet  & \bullet  & \bullet  & \bullet  &\dagger & .  & .\\
       2: & . & . & . &. & . & . & \bullet  & . &.\\
       3: & . & . & . & . & . & . & . & \bullet  & .\\
       4: & . & . & . & . & . & . & . & . & .\\
      5: & . & . & . & . & . & . & . & . & .\\
       6: & . & . & . & . & . & . & . & . & \bullet
       \end{matrix}
\normalsize
\]
For degree reasons alone, entries in the $0^{\th}$ row must lie in the strongly linear strand of the minimal free resolution of $\om_R$, and the length of that strand is $\leq g-1=1$ by \cite[Corollary 1.4]{linear}.  So the first entry that could potentially pose an issue is the one in the position marked by a $\dagger$, as we cannot conclude, for purely degree reasons, that such an entry lies in the strongly linear strand. Let us use the argument in the proof of Lemma~\ref{lem:delicate} to show this entry must be 0.

We adopt the notation of the proof of Lemma~\ref{lem:delicate}. Say we have a cycle $\sigma \in \bigwedge^6 W \otimes \om_R$ corresponding to a nonzero syzygy in position $\dagger$.  For degree reasons, we have $\sigma_i = 0$ for $i \ne 0, 1$; and $\nu(\sigma_1)=0$. In particular, we have $\sigma_1 = x_0x_1\cdots x_5 \otimes y$ for some $y \in (\om_R)_1$. It follows that, for every $f\in W_2^*$, we have $f\sigma_1=0$. The Claim in the proof of Lemma~\ref{lem:delicate} implies that $\nu(\sigma)\ne 0$, and thus $\sigma_0$ must be nonzero and satisfy $\nu(\sigma_0)>0$.  In particular, every term of $\sigma_0$ must involve at least one variable from $W_{>1}$.  We can thus choose an element $f\in W_{>1}^*$ such that $f\sigma_0 \ne 0$.
We therefore have $f\sigma = f\sigma_0 +  f\sigma_1 = f\sigma_0 \ne 0$, which means $f\sigma$ corresponds to a summand of the strongly linear strand that lies in the position of the entry marked $\star$ below:
\[
\footnotesize
\begin{matrix}
          & 0 & 1 & 2 & 3 & 4 & 5 & 6 & 7 & 8\\
       0: & \bullet & \bullet & . & . & . & \star & . & . & .\\
       1: & . &.   &   \bullet  & \bullet  & \bullet  & \bullet  &\dagger & .  & .\\
       2: & . & . & . &. & . & . & \bullet  & . &.\\
       3: & . & . & . & . & . & . & . & \bullet  & .\\
       4: & . & . & . & . & . & . & . & . & .\\
      5: & . & . & . & . & . & . & . & . & .\\
       6: & . & . & . & . & . & . & . & . & \bullet
       \end{matrix}
\normalsize
\]
This is impossible, because the strongly linear strand has length at most $g-1=1$.
\end{example}

\section{Normal generation and the weighted $N_p$ results}\label{sec:proof of Np results}

We will use the notation/assumptions in Setup~\ref{setup:standard} throughout this entire section. Recall that $\phi_W\colon C \to \PP(\d)$ is a closed embedding, by Proposition~\ref{cor:v ample immersion}(2). As above, we denote by $R$ the section ring $\bigoplus_{i \in \Z} H^0(C, L^i)$, and we write $H^0(C,L^a)_{bD}$ for the space of sections of $L^a$ that vanish along the divisor $bD$. We begin with several technical results.

\begin{lemma}
\label{lem:multiple}
We have $(S/I_C)_{\ell d} \cong R_{\ell d}$ for all $\ell \ge 0$. 
\end{lemma}

\begin{proof}
By Proposition~\ref{stackyprop}, we need only show that the ring map $\alpha \colon S \to \bigoplus_i H^0(C, L^{i})$ given by $x_i \mapsto s_i$ induces surjections $\alpha_{\ell d} \colon S_{\ell d} \onto H^0(C, L^{\ell d})$ for all $\ell \ge 0$. By Proposition~\ref{cor:v ample immersion}(1), $\alpha_{d}$ is surjective. Let $V = H^0(C, L^d) $, $f_0, \dots, f_r$ a basis of $V$, and $F_0, \dots, F_r \in S_d$ elements such that $\alpha_d(F_i) = f_i$. Let $\ell \ge 0$. By our assumption on $\deg(L \otimes \OO(-D))$, the embedding $C \into \PP(V)$ determined by $|L^d|$ is normally generated, and so the induced map $h \colon \Sym^\ell(V) \to H^0(C, L^{\ell d})$ is surjective. Let $s \in H^0(C, L^{\ell d})$, and choose $p \in \Sym^\ell(V)$ such that $h(p) = s$, i.e. $p(f_0, \dots, f_r) = s$. We have $\alpha_{\ell d}(p(F_0, \dots, F_r) )= s$. 
\end{proof}

\begin{lemma}\label{lem:multiplicity}
Let $e \geq 0$, and write $e = qd+e’$ for $0 \le e' < d$. 
\begin{enumerate}
\item The natural map
$
H^0(C,L^{qd}) \otimes H^0(C,L^{e'} \otimes \OO(-e'D)) \to H^0(C,L^e \otimes \OO(- e'D))
$
is surjective.

\item The image of the injection $(S/I_C)_e \into H^0(C, L^e)$ is given by the sections that vanish with multiplicity $\geq e’$ along $D$.
\end{enumerate}  
\end{lemma}

\begin{proof}


Part (1) is immediate from \cite[Corollary 4.e.4]{green2}. As for (2): let $\iota$ denote the injection $(S/I_C)_e \into H^0(C, L^e)$.
Because $C$ is embedded by a log complete series of type $(D,d)$, the variables of $S = k[x_0, \dots, x_n]$ have degrees $1$ and $d$. Say $x_0, \dots, x_r $ are the variables of degree 1. Every element of $S_e$, and hence $(S/I_C)_{e}$, lies in $(x_0, \dots,x_r)^{e'}$. It follows that 
every section in the image of $g$ vanishes with multiplicity $\geq e’$ along $D$; that is, 
$\im(\iota) \subseteq H^0(C,L^e)_{e'D}$. 
By Lemma~\ref{lem:H0iso}, there is a natural isomorphism $H^0(C,L^e)_{e'D} \cong H^0(C,L^e \otimes \OO(-e'D))$.  Since $\deg(L \otimes \OO(-D))\geq 2g+1$, the complete linear series on $L \otimes \OO(-D)$ induces a normally generated embedding into projective space, i.e. the natural map $H^0(C,L \otimes \OO(-D))^{\otimes a} \to H^0(C,L^a \otimes \OO(-aD))$ is surjective for all $a\geq 0$. 

We first consider the case where $e < d$, so that $e=e'$ and $q=0$.   We have isomorphisms $(S/I_C)_1= H^0(C,L)_D \cong H^0(C,L \otimes \OO(-D))$ and $H^0(C,L^e \otimes \OO(-eD))\cong H^0(C,L^e)_{eD}$.  Combining these identifications with the surjection $H^0(C,L \otimes \OO(-D))^{\otimes e} \onto H^0(C,L^e \otimes \OO(-eD))$ yields a surjection $\pi \colon (S/I_C)_1^{\otimes e}\onto H^0(C,L^e)_{eD}$. We have a commutative diagram
$$
\xymatrix{
(S/I_C)_1^{\otimes e} \ar[rrd]^-{\pi} \ar[r] & (S/I_C)_e \ar[r]^-{\iota} & H^0(C, L^e) \\
&& H^0(C, L^e)_{eD}, \ar[u]
}
$$
where the vertical map is the inclusion, and the left-most horizontal map is given by multiplication. This proves (2) when $e< d$. Finally, suppose $e\geq d$.  By Lemma~\ref{lem:multiple}, we have $(S/I_C)_{\ell d} \cong R_{\ell d} = H^0(C,L^{\ell d})$ for all $\ell \geq 0$, and we have shown above that $(S/I_C)_{e'} \cong H^0(C,L^{e'} \otimes \OO(-e'D))$.  Part (1) yields a surjection
\[
H^0(C,L^{qd}) \otimes H^0(C,L^{e'} \otimes \OO(-e'D)) \onto H^0(C,L^e \otimes \OO(-e'D))\cong H^0(C,L^e)_{e'D}.
\]
Combining these observations, we see that there is a surjection $\pi : (S/I_C)_{qd}\otimes (S/I_C)_{e'} \onto H^0(C,L^e)_{e'D}$ such that the diagram
$$
\xymatrix{
(S/I_C)_{qd}\otimes (S/I_C)_{e'} \ar[rrd]^-{\pi} \ar[r] & (S/I_C)_e \ar[r]^-{\iota} & H^0(C, L^e) \\
&& H^0(C, L^e)_{e'D} \ar[u]
}
$$
commutes, where the vertical map is the inclusion, and the left-most horizontal map is multiplication. The result follows.
\end{proof}

\begin{prop}\label{prop:conductor}
Let $Q$ denote the cokernel of the injection $S/I_C \into R$.
\begin{enumerate}
\item  We have $Q_{qd}=0$ for all $q \ge 0$.  In particular, if $0\leq j <d$, then any element of $S_{d-j}$ annihilates any element of $Q_{qd+j}$.
\item  For all $e\geq 0$, we have $\dim Q_e = \dim Q_{e+d}$.
\item  The support of the sheaf $\widetilde{Q}$ is the set of points in $D$.  In particular, $\widetilde{Q}$ is a $0$-dimensional sheaf on $\PP(W)$, and $Q$ is a $1$-dimensional $S$-module.
\item $H^j_{\mathfrak m} Q = 0$ for $j\ne 1$, and $(H^1_{\mathfrak m} Q)_e= 0$ for $e\geq 0$.  In particular, $Q$ is a Cohen-Macaulay $S$-module, and its weighted regularity (Definition~\ref{defn:benson}) is at most 0.
\end{enumerate}
\end{prop}

Before beginning the proof, we discuss a simple example. 

\begin{example}
Consider Example~\ref{ex:P1 deg2}, where $S/I_C \cong k[s^2,st,st^3,t^4]$ and $R \cong k[s^2,st,t^2]$, so that $Q=t^2 \cdot k [t^4]$.  In other words, letting $M = S / (x_0, x_1, x_2)$, we have $Q \cong M(-1)$. Observe that $Q$ is concentrated in positive odd degrees, and each of its nonzero homogeneous components is a 1-dimensional $k$-vector space. Its support is the point $V(x_0, x_1, x_2)$ in $\PP(W)$, which is the point in $D$. Clearly $H^0_{\mathfrak m} Q=0$, because $x_3$ is a non-zero-divisor on $Q$. A local duality argument implies that $H^1_{\mathfrak m} Q= t^{-2}\cdot k[t^{-4}]$, which is zero in nonnegative degrees.
\end{example}

\begin{proof}[Proof of Proposition~\ref{prop:conductor}]
Part (1) is clear from Lemma~\ref{lem:multiple}, and Part (3) is immediate from Lemma~\ref{lem:multiplicity}(2). For (2): we write $e=qd+e'$ with $0\leq e'<d$ and $q\geq 0$. When $e = 0$, this is clear from (1), so assume $e > 0$. We have
\[
\dim Q_e= \dim H^0(C,L^e) - \dim H^0(C,L^e \otimes \OO(-e'D)) =e'\cdot \deg D,
\]
where the first equality follows from Lemmas~\ref{lem:H0iso} and \ref{lem:multiplicity}(2), and the second follows from Riemann-Roch. In particular, we see that $\dim Q_e$ only depends on the remainder of $e$ modulo $d$; this proves (2).
Finally, we consider (4).  The inclusion $D\subseteq C$ yields a short exact sequence
$$
0\to \cO_C(-D)\to \cO_C\to \cO_D\to 0.
$$ 
Twisting by $L^d$, and noting that $\cO_D \otimes L^d =\cO_D$ because $D$ is 0-dimensional, we obtain a short exact sequence
$$
0\to L^d \otimes \cO_{C}(-D) \to L^d \to \cO_D\to 0.
$$  
Noting that $H^1(C,L^d \otimes \OO(-D))=0$ since $\deg(L^d \otimes \OO(-D)) \ge \deg(L \otimes \OO(-D)) \ge 2g+1$, this short exact sequence induces a surjection 
\begin{equation}
\label{eqn:surj}
(S/I_C)_d \cong H^0(C,L^d) \onto H^0(D,\cO_D).
\end{equation}
Since $D$ is a finite collection of points, it is an affine scheme, and so $H^0(D,\cO_D)$ contains a unit.  Choose a degree $d$ element $u\in S_d$ such that the surjection
$$
S_d \onto (S/I_C)_d \xra{\eqref{eqn:surj}} H^0(D,\cO_D)
$$
sends $u$ to a unit. This implies that the map $Q\to Q(d)$ given by multiplication by $u$ does not alter the multiplicity of vanishing along $D$ and thus induces an isomorphism $Q_e\to Q_{e+d}$ for all $e\geq 0$.  In particular, any nonzero element of $Q_e$ with $0<e<d$ cannot be annihilated by the entire maximal ideal $\mathfrak m$, and so $H^0_{\mathfrak m} Q=0$.  Since $\dim Q=1$, we also have $H^i_{\mathfrak m}Q = 0$ for $i>1$. It remains to consider $H^1_{\mathfrak m} Q$. Using that $H^0_{\mathfrak m} Q=0$, \cite[Theorem A4.1]{eisenbudbook} yields a short exact sequence
$$
0 \to Q_e \to H^0(\PP(W),\widetilde{Q(e)}) \to  (H^1_{\mathfrak m} Q)_e \to  0.
$$
We know $\dim Q_e = \dim Q_{e + d}$ for all $e \ge 0$. In fact, since the map $Q(e) \xra{u} Q(e + d)$ is injective and has finite dimensional cokernel for all $e \in \Z$, we have $\widetilde{Q(e)} \cong \widetilde{Q(e+d)}$ for all $e \in \Z$. It follows that $(H^1_{\mathfrak m} Q)_e \cong (H^1_{\mathfrak m} Q)_{e+d}$ for all $e\geq 0$.  However, $(H^1_{\mathfrak m} Q)_e = 0$ for $e\gg 0$, and so we must have $(H^1_{\mathfrak m} Q)_e = 0$ for all $e\geq 0$.
\end{proof}

\begin{proof}[Proof of Theorem~\ref{thm:proj norm}]
From the short exact sequence
$
0\to S/I_C \to R\to Q\to 0,
$ we get a long exact sequence in local cohomology.  Since $H^0_{\mathfrak m} Q = 0$ by Proposition~\ref{prop:conductor}, and $H^1_{\mathfrak m} R=0$ by  Lemma~\ref{lem:technical}(1), we conclude that $H^1_{\mathfrak m} (S/I_C)=0$.
Thus, $S/I_C$ is normally generated, and it follows from Remark~\ref{rem:proj normal}(1) that $S/I_C$ is Cohen-Macaulay.
\end{proof}

\begin{cor}\label{cor:reg SI}
The weighted regularity of $S/I_C$ and $R$ is $2$ if $g>0$ and $1$ if $g=0$.
\end{cor}
\begin{proof}
By Theorem~\ref{thm:proj norm}, $S/I_C$ is Cohen-Macaulay, and so $H^0_{\mathfrak m} (S/I_C) =H^1_{\mathfrak m} (S/I_C)=0$. Since $R$ is a Cohen-Macaulay $S$-module by Lemma \ref{lem:technical}(1), and $Q$ is a 1-dimensional $S$-module by Proposition~\ref{prop:conductor}(3), the short exact sequence $0\to S/I\to R\to Q\to 0$ yields the short exact sequence 
$
0\to H^1_{\mathfrak m} Q \to H^2_{\mathfrak m} (S/I_C) \to H^2_{\mathfrak m} R \to 0.
$
Proposition~\ref{prop:conductor}(4) implies that $(H^1_{\mathfrak m} Q)_e=0$ for $e\geq 0$, and \eqref{eqn:isos} implies $(H^2_{\mathfrak m} R)_e\cong H^1(C,L^e)$.  We have $H^1(C,L^e) = 0$ if and only if $e>0$ (resp. $e \ge 0)$ when $g>0$ (resp. $g=0$). The statement immediately follows.
\end{proof}

\begin{proof}[Proof of Theorem~\ref{thm:main}]
Normal generation follows from Theorem~\ref{thm:proj norm}. Let us record the following computation:
\begin{align*}
\dim W &= \dim W_1 + \dim W_d \\
&= \dim H^0(C, L \otimes \OO(-D)) + \dim H^0(C, L^d) - \dim H^0(C, L^d \otimes \OO(-dD)) \\
&= \deg(L \otimes \OO(-D)) - g + 1 + d\deg(D)\\
&\ge  g + 2 + q + d\deg(D),
\end{align*}
where the first two equalities follows from the definition of a log complete series along with Lemma~\ref{lem:H0iso}, the third from Riemann-Roch, and the inequality by hypothesis.  Also,
$$
\dim W_1 = \dim H^0(C, L \otimes \OO(-D)) = \deg(L \otimes \OO(-D)) - g + 1 \ge g + 2 + q > g, 
$$
and so the assumption $\dim S_1 > g$ in Theorem~\ref{thm:virtualNp} holds here.

Let $Q$ be as in Proposition~\ref{prop:conductor}.
By Theorem~\ref{thm:proj norm}, Lemma~\ref{lem:technical}(1), and Proposition~\ref{prop:conductor}(4), we have a short exact sequence
$
0\to S/I_C \to R \to Q \to 0
$
of Cohen-Macaulay $S$-modules of dimensions $2,2,$ and $1$, respectively. Recall that $S=k[x_0, \dots, x_n]$ and $|\d| = \sum_{i=0}^n \deg x_i$.   Write $\omega _{R}\coloneqq \Ext_S^{n-1}( R, S(-|\d|))$,  $\omega_{S/I_C}\coloneqq \Ext_S^{n-1}( S/I_C, S(-|\d|))$ and $\omega_Q\coloneqq\Ext_S^{n}(Q,S(-|\d|))$ for the Matlis duals of these modules.  We have a short exact sequence
\[
0\to\omega_R \to \omega_{S/I_C}\to\omega_Q \to 0.
\]
Just as in our proof of Theorem~\ref{thm:virtualNp}, we must consider the $g=0$ and $g > 0$ cases separately:
\vskip\baselineskip
\noindent \emph{Case 1: $g = 0$.} While we argue as in the proof of Theorem~\ref{thm:virtualNp}, we recapitulate the details for completeness. Since $S/I_C$ is Cohen-Macaulay of dimension 2, we have 
$
\beta_{i,j}(S/I_C) = \beta_{n - 1 - i,  |\mathbf{d}| - j}(\om_{S/I_C})
$ for all $i, j$.
Now, suppose $\beta_{i, j}(S/I_C) = \beta_{n-1-i,  |\mathbf{d}| - j}(\om_{S/I_C}) \ne 0$, and assume $j > w^{i+1}$. We have:
$$
|\mathbf{d}| - w_{n-i} = \sum_{j=0}^n d_j - \sum_{j=0}^{n-1-i} d_j = \sum_{j=n-i}^n d_j = w^{i+1} < j.
$$
Rearranging, we get $|\d| - j  < w_{n-i}$. Corollary~\ref{cor:reg SI} (along with Local Duality) implies that $(\om_{S/I_C})_1 \ne 0$ and $(\om_{S/I_C})_{< 1} = 0$. Since $I_C$ is prime, every variable $x_i \in S$ is a non-zero-divisor on $\om_{S/I_C}$. In particular, $x_0$ has this property. Applying Lemma~\ref{lem:easy}(2), with $\ell = 0$, we get $|\d| -j\ge w_{n-i}$, a contradiction. Thus, if $\beta_{i, j}(S/I_C) \ne 0$, then $j \le w^{i+1}$. It follows that the embedding $C \subseteq \PP(W)$ satisfies the weighted $N_p$-condition for all $p \ge 0$.
\vskip\baselineskip
\noindent \emph{Case 2: $g > 0$.} 
 We first prove that
 \begin{equation}
 \label{eqn:tor}
  \Tor_i^S(\omega_R,k)_{j}=  \Tor_i^S(\omega_{S/I_C},k)_{j}
 \end{equation}
 for all $j < w_{i+1}$. Proposition~\ref{prop:conductor}(4) (along with Local Duality) implies that $(\om_{Q})_i = 0$ for $i \le 0$, while Corollary~\ref{cor:reg SI} (along with Local Duality) implies that $(\om_{S/I_C})_0 \ne 0$ and $(\om_{S/I_C})_{< 0} = 0$, and similarly for $\om_R$. Lemma~\ref{lem:strongly linear strands} therefore implies that the strongly linear strands of the minimal free resolutions of $\om_R$ and $\om_{S/I_C}$ are isomorphic. The identification \eqref{eqn:tor} now follows by applying Lemma~\ref{lem:delicate} to both $\om_R$ and $\om_{S/I_C}$.
(Note that $\dim (\om_{R})_0 = \dim (\om_{S/I_C})_0 = g$, and so, since $\dim W_1 > g$, the assumption ``$\dim W_1 > M_0$" in Lemma~\ref{lem:delicate} holds for both $M = \om_R$ and $M = \om_{S/I_C}$.) Finally, as in the proof of Theorem~\ref{thm:virtualNp} (and Case 1), we have $\b_{i,j}(R) = \b_{n-1-i, |\d| - j}(\om_R)$, and similarly for $S/I_C$. The equality \eqref{eqn:tor} implies that $\Tor_{i}(R,k)_{j} =  \Tor_i(S/I_C,k)_{j}$ whenever $|\d| - j < w_{n - i},$ i.e. $j > |\d| - w_{n-i} = w^{i+1}$. Applying Theorem~\ref{thm:virtualNp}, we therefore conclude that, if $i \le \dim W - g - 2$ and $\b_{i,j}(S/I_C) \ne 0$, then $j \le w^{i+1}$. Since  $\dim W - g - 2 \ge q + d\deg(D)$, it follows that the embedding $C \subseteq \PP(W)$ satisfies the weighted $N_{q + d\deg(D)}$ property.
  \end{proof}

\begin{proof}[Proof of Corollary~\ref{cor:quadrics}]
Immediate from Theorem~\ref{thm:main}.
\end{proof}

\section{Questions}\label{sec:questions}
\subsection{Higher dimensional varieties}

Mumford famously showed that any high degree Veronese of a projective variety is  ``cut out by quadrics''~\cite{mumford-quadrics}; see also the generalization in~\cite{sidman-smith}.  Corollary~\ref{cor:quadrics} is an analogue of Mumford's result for curves in weighted projective spaces; it is natural to ask if this result can be extended to other varieties in weighted projective spaces:
\begin{question}\label{q:mumford}
Can one prove results like Corollary~\ref{cor:quadrics} for higher dimensional varieties embedded in weighted projective spaces?
\end{question}

We can also ask about normal generation and the $N_p$-conditions for higher dimensional varieties.  Here, the central results are those of ~\cite{ein-lazarsfeld-np}, which prove $N_p$-results for embeddings by line bundles of the form $K_X+L^d+B$, where $K_X$ is the canonical bundle, $L$ is very ample, and $B$ is effective.

\begin{question}
Can one obtain $N_p$-conditions for higher dimensional varieties embedded by log complete series, under hypotheses similar to those in~\cite{ein-lazarsfeld-np}?  
\end{question}

Embeddings into weighted spaces also provide an intermediate case for investigating asymptotic syzygy-type questions, as in~\cite{ein-lazarsfeld-asymptotic}.

\begin{question}
With notation as in Question~\ref{q:mumford}, can one prove asymptotic nonvanishing results, similar to what happens in the main results of~\cite{ein-lazarsfeld-asymptotic}?   At the other extreme, can one prove asymptotic vanishing results as in \cite{park}? 
\end{question}

\subsection{Scrolls and the Gonality Conjecture}
There is a rather trivial sense in which curves embedded via log complete series of high degree satisfy an analogue of Green-Lazarsfeld's Gonality Conjecture.  Recall that a high degree curve in $\PP^r$ has regularity $2$, and so the Betti table looks as follows:
\[
\bordermatrix{
&0&1&2&\cdots& a&a+1&\cdots &b&b+1&\cdots \cr
0&*&-&-&\cdots&-&-&\cdots&-&-&\cdots \cr
1&-&*&*&\cdots&*&*&\cdots&*&-&\cdots \cr
2&-&-&-&\cdots&-&*&\cdots&*&*&\cdots \cr
}
\]
The $N_p$-conditions are about the moment we first get nonzero entries in row $2$, i.e. column $a+1$ in the picture.  In~\cite{GL}, Green-Lazarsfeld conjectured that the moment where we first get a zero entry in row $1$, i.e. column $b+1$ in the picture, is determined by the gonality $\text{gon}(C)$ of the curve.  This is the Green-Lazarsfeld Gonality Conjecture, and it was proven in~\cite{ein-lazarsfeld-gonality}, utilizing techniques originally developed by Voisin~\cite{voisin-even}.

In the standard graded setting, $b$ is the maximal index such that $F_i$ has a minimal generator of degree $i+1$.  In the weighted setting, a natural analogue of the invariant $b$ would be to let  
$b(C) \coloneqq \max\{ i \text{ : } F_i \text{ has a generator of degree $w_i+1$}\}.$  However, the main result of \cite{ein-lazarsfeld-gonality} immediately implies that, with notation as in Theorem~\ref{thm:main}, we have that $b(C) = \dim W_1 - 2 -\on{gon}(C)$ for $\deg L \gg 0$.
Since this only depends on the degree one part of $W$, it tells us nothing new about the relationship between the geometry of curves and the algebraic properties of syzygies.  
So if we want to find a meaningful  weighted analogue of the Gonality Conjecture, we will need to look in a different direction.

The Green-Lazarsfeld Gonality Conjecture is one of a series of conjectures about the extent to which the syzygies of a curve $C$ are determined by embeddings of $C$ into special varieties such as scrolls or other varieties of minimal degree (or minimal regularity).  To develop a meaningful weighted analogue of the Green-Lazarsfeld Gonality Conjecture, a natural first question to tackle would be:
 
 \begin{question}\label{q:scrolls}
Can we develop a weighted theory of rational normal scrolls, or varieties of minimal degree, or varieties of minimal regularity?  More specifically, can one develop such theories for the weighted spaces $\PP(1^a,d^b)$ that arise in Theorem~\ref{thm:main}?
\end{question}

There is a famous classical connection between varieties of minimal degree and the $N_p$-condition: a variety has minimal degree if and only if it satisfies the $N_p$-condition for the maximal possible $p$, i.e. if and only if its resolution is purely linear.  If one can answer parts of Question~\ref{q:scrolls}, it would be interesting to then investigate how that answer is related to the weighted $N_p$-conditions explored in this paper.

In a different direction, a famous result of Gruson-Lazarsfeld-Peskine \cite{GLP} bounds the regularity of any nondegenerate irreducible curve $C\subseteq \PP^r$ in terms of its degree.   It would be interesting to explore an analogue of such a theorem:

\begin{question}
Let $C$ be a smooth (or irreducible) curve in $\PP(d_0, \dots, d_n)$.  Can one bound the regularity of $I_C$ via a Gruson-Lazarsfeld-Peskine-type formula?
\end{question}

\subsection{$M_L$-bundles}\label{subsec:ML}
Green, Lazarsfeld and others have used positivity of $M_L$-bundles to obtain $N_p$-results for syzygies of a curve $C$ embedded by a line bundle $L$~\cite{AKL,ein-lazarsfeld-np,GL, GLP, KL,pareschi,park}. Let $C\subseteq \PP^n$ be a curve embedded by the complete linear series for  $L$. The vector bundle $M_L$ is defined by the short exact sequence:
$
0\gets L \gets H^0(C,L)\otimes_k \cO_C \gets M_L \gets 0.
$
Vanishing results about exterior powers of $M_L$ can be used to obtain $N_p$-results about syzygies of the embedding $C\into \mathbb P^r$ by the complete linear series $|L|$.  

In the nonstandard graded case, the setup is more subtle, as the linear series involves sections of different degrees.  This would require altering the basic framework, and 
it would be interesting to see whether $N_p$-results for varieties could be proven via weighted analogues of this approach.

\subsection{Stacky curves}
Stacky curves have arisen in recent work on Gromov-Witten theory, mirror symmetry, the study of modular curves and more; see~\cite{stackycurve} and the references within.
In~\cite{stackycurve}, Voigt and Zureick-Brown
 prove analogues of classical results like Petri's Theorem for stacky curves; in fact, their results can be viewed as showing that the canonical embeddings satisfy our weighted $N_1$-condition. Stacky curves cannot generally be embedded into standard projective space; rather, they embed into weighted projective stacks. The only relevant $N_p$-conditions for such curves are therefore in the weighted projective setting.

\begin{question}
Prove an analogue of Theorem~\ref{thm:main} for stacky curves embedded into weighted projective stacks by high degree line bundles.  
\end{question}
We highlight one aspect where stacky curves differ from smooth curves.  For a line bundle of high enough degree on a smooth curve, the rank of the global sections depends only on the degree of the line bundle.  This is not the case for stacky curves, as the space of global sections also depends on the behavior of the corresponding divisor at the stacky points.  So instead of simply fixing the degree of the line bundle, a more natural setup for a stacky curve might be to follow the recipe from~\cite{ein-lazarsfeld-np} and focus on $N_p$-conditions for divisors of the form $K+ L^d + B$, where $K$ is the canonical divisor, $L$ is very ample, and $B$ is effective.

In a slightly different direction, one could focus on canonical embeddings.  Green's Conjecture relates classical $N_p$-conditions to the intrinsic geometry of a canonical curve, specifically to its Clifford index.  The canonical embedding of a stacky curve lands in a weighted projective stack, and thus our weighted $N_p$-conditions provide a natural setting for considering a stacky analogue of Green's Conjecture:
\begin{question}
Can one use weighted $N_p$-conditions to articulate an analogue of Green's Conjecture for stacky curves?
\end{question}

\subsection{Nonstandard Koszul rings}
Koszul rings were defined by Priddy~\cite{priddy} and now play a fundamental role within commutative algebra~\cite{avramov-eisenbud, avramov-peeva, aci, conca}. 
One rich source of Koszul rings comes from high degree Veronese embeddings.  
Let $X\subseteq \PP^r$ be a smooth variety embedded by a complete linear series for $L^d$, where $L$ is very ample and $d\gg 0$; it is known that the homogeneous coordinate ring of $X\subseteq \PP^r$ is a Koszul ring~\cite{backelin,ERT}.

It would be interesting to know if high degree embeddings into weighted spaces (via log complete series) can provide more exotic examples of Koszul rings, or related concepts.  The following example shows some of the subtle behavior that might arise.

\begin{example}
Let $\PP^1 \to \PP(1^3,2^2)$ be the map $[s:t]\mapsto [s^3:s^2t:st^2:st^5:t^6]$ be the map determined by the log complete series for $\cO_{\PP^1}(3)$ with $d=2$ and $\deg(D)=1$.  Let $S=k[x_0,x_1,x_2,y_0,y_1]$ be the Cox ring of $\PP(1^3,2^2)$.  The defining ideal of the image is
$
I = \langle 
x_{1}^{2}-x_{0}x_{2},\,x_{2}y_{0}-x_{1}y_{1},\,x_{
      1}y_{0}-x_{0}y_{1},\,x_{2}^{3}-x_{0}y_{1},\,x_{1}x_{2}^{2}-x_{0}y_{0},\,y
      _{0}^{2}-x_{2}^{2}y_{1}
\rangle.
$
The ring $S/I_C$ is isomorphic to the subalgebra $k[s^3,s^2t,st^2,st^5,t^6]$, and it is a variant of a Veronese subring; for instance, it contains the degree $6$ Veronese subring.  The ring $T = S/I_C$ does not satisfy the standard definition of a graded Koszul ring, as the minimal free resolution of the residue field has the form
$
\left[T \gets T(-1)^3 \oplus T(-2)^2 \from \cdots\right].
$
However, if we consider the grevlex order with $y_0>y_1>x_0>x_1>x_2$, then the initial ideal is
$
\operatorname{in}(I) = \langle x_{1}^{2},\,x_{1}y_{1},\,x_1y_{0},\,x_0y_{1},\,y_{0}x_{0},\,y_{0}^{2}\rangle.
$
Thus, $I$ has a quadratic Gr\"obner basis; if $S$ were standard graded, then this would imply that $S/I_C$ is $G$-quadratic and therefore Koszul~\cite{conca}. Given the nonstandard grading, it implies that $S/I_C$ is a sort of nonstandard graded deformation of a Koszul ring.
\end{example}
\begin{question}
Let $X$ be a smooth variety, and consider an embedding $X\into \PP(W)$ given by a log complete series for $L^e$, where  $L$ is very ample and $e \gg 0$.  Let $I_X\subseteq S$ be the defining ideal in the corresponding nonstandard graded polynomial ring.  Does $I_X$ admit a quadratic Gr\"obner basis?  What sort of Koszul-type properties are satisfied by $S/I_X$?
\end{question}

\subsection{$N_p$-conditions for curves in other toric varieties}\label{subsec:other toric}
Another natural direction is to ask whether smooth curves in other toric varieties also satisfy $N_p$-conditions. To approach this, one must consider:

\begin{question}\label{q:np toric}
Let $S$ be the $\ZZ^r$-graded Cox ring of a simplicial toric variety $X$ and $B$ the corresponding irrelevant ideal.  
\begin{enumerate}
	\item What is a good analogue of complete, or log complete, linear series?
	\item  What is a good analogue of normal generation?
	\item What is the appropriate analogue of the $N_p$-conditions in this setting?
	\item Does the answer to (1) or (2) depend only on the grading of $S$, or does it also depend on the choice of the irrelevant ideal $B$?  
	\item When defining the $N_p$-conditions, should one focus on minimal free resolutions or on virtual resolutions?
\end{enumerate}
\end{question}
Even for a product of projective spaces, some of these questions are open: 
\begin{question}
Can one develop analogues of the main results of this paper for a smooth curve in a product of projective spaces?
\end{question}

\bibliographystyle{amsalpha}
\bibliography{Bibliography}

\end{document}